\documentclass[reqno,11pt]{amsart}
\usepackage{a4}
\usepackage{latexsym}
\usepackage{amssymb,latexsym,amscd,amsmath,amsthm}
\usepackage{hyperref}

\usepackage{tikz}
\usetikzlibrary{arrows}
\newcommand\C{\ensuremath{\mathbb C}}
\renewcommand\H{\ensuremath{\mathbb H}}
\newcommand\Q{\ensuremath{\mathbb Q}}
\newcommand\R{\ensuremath{\mathbb R}}
\newcommand\Z{\ensuremath{\mathbb Z}}

\let\eps=\varepsilon

\let\into=\hookrightarrow
\newcommand\rk{\operatorname{rk}}
\newcommand\ind{\operatorname{ind}}
\newcommand\im{\operatorname{im}}

\newcommand\del{\partial}
\newcommand\norm[1]{\left\|#1\right\|}

\newcommand\Hom{{\operatorname{Hom}}}
\newcommand\End{{\operatorname{End}}}

\newcommand\ch{\operatorname{ch}\nolimits}

\newcommand\ek{\mu}

\newcommand\arsinh{\operatorname{ar\mskip 2mu sinh}}
\newcommand\Bun{\operatorname{Bun}}
\newcommand\conn{\mathbin{\sharp}}
\newcommand\bconn{\mathbin{\natural}}

\newcommand\Top{{\text{\it Top\/}}}
\newcommand\Cat{{\text{\it Cat\/}}}

\newcommand{\SU}{\mathrm{SU}}

\newcommand{\Spin}{\mathrm{Spin}}

\newcommand{\spin}{\mathrm{spin}}

\newcommand{\U}{\mathrm{U}}

\newcommand{\KS}{\mathrm{KS}}

\numberwithin{equation}{section}

\swapnumbers
\theoremstyle{plain}
\newtheorem{Lemma}{Lemma}
\numberwithin{Lemma}{section}
\newtheorem{Proposition}[Lemma]{Proposition}

\newtheorem{Corollary}[Lemma]{Corollary}

\newtheorem{Theorem}[Lemma]{Theorem}

\theoremstyle{definition}
\newtheorem{Definition}[Lemma]{Definition}

\theoremstyle{remark}

\newtheorem{Remark}[Lemma]{Remark}
\newtheorem{Example}[Lemma]{Example}

\begin{document}

\title[Quaternionic Line Bundles]{Kreck-Stolz invariants for quaternionic line bundles} 

\author{Diarmuid Crowley and Sebastian Goette}

\subjclass[2000]{Primary: 58J28, 57R55, 57R20}

\keywords{$\eta$-invariant, quaternionic line bundle, $7$-manifold, smooth structure, Kirby-Siebenmann invariant}

\address{Hausdorff Research Institute for Mathematics\\
Universit\"at Bonn\\
Poppelsdorfer Allee 82\\
D-53115 Bonn \\
Germany} \email{diarmuidc23@gmail.com}

\address{Mathematisches Institut\\ Universit\"at Freiburg\\ Eckerstr.~1, 79104 Freiburg, Germany} \email{sebastian.goette@math.uni-freiburg.de}

\begin{abstract}
We generalise the Kreck-Stolz invariants $s_2$ and $s_3$ by defining a new invariant, the $t$-invariant, for quaternionic line bundles over closed spin-manifolds $M$ of dimension $4k-1$ with $H^3(M; \Q) = H^4(M; \Q) = 0$.  The $t$-invariant classifies closed smooth oriented $2$-connected rational homology $7$-spheres up to almost-diffeomorphism and detects exotic homeomorphisms between such manifolds.  

The $t$-invariant also provides information about quaternionic line bundles over a fixed manifold and we use it to give a new proof of a theorem of Feder and Gitler about the values of the second Chern classes of quaternionic line bundles over $\H P^k$.  The $t$-invariant for $S^{4k-1}$ is closely related to the Adams $e$-invariant on the $(4k-5)$-stem.
%
%
%
\end{abstract}

\maketitle

\section*{Introduction}

In~\cite{KS}, Kreck and Stolz introduced invariants~$s_1$, $s_2$,
$s_3\in\Q/\Z$ of certain closed smooth oriented simply connected~$7$-manifolds~$M$ that completely characterise~$M$  
up to diffeomorphism.  The $s$-invariants are defect invariants based on index theorems for Dirac operators on $8$-manifolds: the invariant $s_1$, which equals the Eells-Kuiper invariant~$\mu$ of~\cite{EK} if~$M$ is spin, is the defect of the untwisted Dirac operator, whereas~$s_2$ and~$s_3$ are the defects of the Dirac operator twisted by certain {\em complex line bundles}.

In this paper we define a defect invariant, the $t$-invariant, based on index theorems for Dirac operators twisted by {\em quaternionic line bundles}.  Suppose now that~$M$ is a closed smooth spin $(4k-1)$-manifold such that the groups~$H^3(M; \Q)$ and~$H^4(M; \Q)$ vanish and let~$\Bun(M)$
%
%
denote the set of isomorphism classes of quaternionic line bundles over~$M$.  The $t$-invariant is a function (see Definition~\ref{B2.D1}),
\[ t_M \colon \Bun(M) \longrightarrow \Q/\Z \; , \]
which as we explain in Section~\ref{C3}, is a precise generalisation of the
invariants~$s_2$ and~$s_3$ in dimension~$7$.
Another important invariant of quaternionic line bundles is their second Chern class which defines a function
%
\[ c_2 \colon \Bun(M) \longrightarrow H^4(M) \,.\]
We shall say that~$M$ and~$N$ have {\em isomorphic} $t$-invariants if there is a group isomorphism~$A \colon H^4(M) \to H^4(N)$ and a set bijection~$B \colon \Bun(M) \to \Bun(N)$ which are compatible with second Chern class and the $t$-invariant.


Our first application of the $t$-invariant is in the realm of $2$-connected $7$-manifolds and was inspired by recent discoveries of Grove, Verdiani and Ziller~\cite{GVZ}. The first named author established a complete classification of $2$-connected rational homology $7$-spheres in~\cite{Crow} using a certain {\em extrinsically defined} quadratic linking form~$q_M \colon H^4(M) \to \Q/\Z$ whose values are calculated using a spin $8$-manifold~$W$ with boundary~$\partial W = M$.  Theorem \ref{Z.T1} below, which is a reformulation of Theorem \ref{C1.T2}, states that the $t$-invariant $t_M$ is a refinement of~$q_M$.  Moreover, as a defect invariant, the $t$-invariant has an {\em intrinsic definition}, see \eqref{B2.4}, given via $\eta$-invariants arising from the Atiyah-Patodi-Singer index theorem of~\cite{APS}. 

\begin{Theorem} \label{Z.T1}
Let $M$ be a closed smooth $2$-connected oriented rational homology $7$-sphere.  Then 
\begin{enumerate}
\item \label{Z.T1.1} the map $c_2 \colon \Bun(M) \to H^4(M)$ is onto,
\item \label{Z.T1.2} for all $E \in \Bun(M)$, $q_M(c_2(E)) = 12 t_M(E)$,
\item \label{Z.T1.3} the map $c_2 \times t_M \colon \Bun(M) \to H^4(M) \times \Q/\Z$ is injective.
\end{enumerate}
\end{Theorem}
\noindent 
Applying \cite[Theorem A]{Crow} we obtain the following expanded version of Corollary \ref{C1.C1}.

\begin{Corollary} \label{Z.C1}
Let~$N$ and~$M$ be closed smooth $2$-connected oriented rational homology $7$-spheres.  Then~$N$ is diffeomorphic to~$M$ if and only if~$t_N$ is isomorphic to~$t_M$ and the Eells-Kuiper invariants of~$M$ and~$N$ agree: $\mu(N) = \mu(M)$.
\end{Corollary}
\noindent
In \cite{Goe} the second named author used Corollary \ref{Z.C1} to classify manifolds constructed in~\cite{GVZ} by directly calculating their $t$-invariants as defined in \eqref{B2.4},
see example~\ref{C1.E1}.

A further feature of the $t$-invariant in dimension~$7$ is that it detects {\em exotic homeomorphisms}: these are homeomorphisms~$h \colon N \cong M$ which are not homotopic to piecewise linear homeomorphisms.   The work of Kirby and Siebenmann~\cite{KiSi} implies that for a homeomorphism~$h \colon N \cong M$ there is an invariant 
\[ \KS(h) \in H^3(M; \Z/2) \;, \]
depending only on the homotopy class of~$h$, such that~$h$ is exotic if and only if~$\KS(h) \neq 0$.  The following shorter version of Theorem \ref{C2.T1} combined with Theorem \ref{Z.T1}~\eqref{Z.T1.1} above shows that  the Kirby-Siebenmann invariant of $h$ can be computed using the induced map $h^* \colon \Bun(M) \to \Bun(N)$ and the $t$-invariants of $N$ and $M$.

\begin{Theorem} \label{Z.T2}
A homeomorphism~$h: N \to M$ is exotic if and only if
\[t_M \neq t_N \circ h^*.\]
More precisely, for all~$E \in \Bun(M)$,  
\[ \bigl(\KS(h) \smile c_2(E) \bigr) [M]_2  =  t_N(h^*E) - t_M(E) \in \Z/2 \subset \Q/\Z \;, \]
where~$[M]_2$ generates~$H_7(M; \Z/2)$. 
%
\end{Theorem}

We now change our focus from the base-space manifolds to the bundles themselves.  For the simplest manifolds~$M = S^{4k-1}$ we have the $t$-invariant
\[ t_{S^{4k-1}} \colon \Bun(S^{4k-1}) \cong \pi_{4k-2}(S^3) \longrightarrow \Q/\Z \; . \]
Given its relationship to the Dirac operator and hence the $\widehat A$-genus one might expect that the $t$-invariant is related to the Adam's $e$-invariant~$e \colon \pi_{4k-5}^S \to \Q/\Z$ of~\cite{A}.  Theorem \ref{B3.T1}, restated immediately below, bears this expectation out.

\begin{Theorem}\label{Z.T3}
Assume that~$k\ge 3$.
For any homotopy sphere~$\Sigma^{4k-1}$ the homomorphism
\[ t_\Sigma \colon \Bun(\Sigma) \cong \pi_{4k-2}(S^3) \longrightarrow \Q/\Z\]
may be identified with the composition
\[ - e \circ S \colon \pi_{4k-2}(S^3) \longrightarrow \pi_{4k-5}^S \longrightarrow \Q/\Z\]
where~$- e \colon \pi_{4k-5}^S \to \Q/\Z$ and~$S \colon \pi_{4k-2} \to \pi_{4k-5}^S$ are respectively the the negative of the Adams $e$-invariant~\cite{A} and the stabilisation homomorphism.
\end{Theorem}

%
%

A key feature of quaternionic bundles is that both the quaternions~$\H$ and their group of units~$S^3 \subset \H$ are non-abelian.  As a result the classifying space~$BS^3 = \H P^\infty$ is not an H-space, and for a general space~$X$, $\Bun(X) \equiv [X, BS^3]$ does not have a naturally defined group structure.  This leads to the fact that it is often a difficult problem to determine the image of~$c_2$.  For the case $X = \H P^k$ and the map
\[ c_2 \colon \Bun(\H P^k) \longrightarrow H^4(\H P^k) \]
one says that an integer $c$ is $k$-realisable if $c \cdot c_2(H) \in {\rm Im}(c_2)$ where $H$ is the tautological bundle on $\H P^k$ and $c_2(H)$ is the universal second Chern class.  In~\cite{FG} Feder and Gitler proved

\begin{Theorem}[{\cite[Theorem 1.1]{FG}}] \label{Z.T4}
If an integer~$c$ is $k$-realisable then
  \begin{equation*}
    \frac{a_j}{(2j)!}\prod_{i=0}^{j-1}(c-i^2) = 0 \in\Q/\Z
	\qquad\text{for all }2\le j\le k\;.
  \end{equation*}
\end{Theorem}
\noindent
Using the $t$-invariant, we give a new proof of the Feder-Gitler criteria of Theorem~\ref{Z.T4} based on Theorem~\ref{D1.T1}.

The rest of the paper is organised as follows: in Section~\ref{B} we define the $t$-invariant from both the intrinsic analytic point of view and the extrinsic topological point of view.  This section also introduces the basic concepts and tools required for our analysis of quaternionic line bundles. 
In section~\ref{C} we consider the $t$-invariant for $7$-manifolds.  We up-date the classification results of~\cite{Crow} in Theorem~\ref{C1.T2}, we show that the $t$-invariant detects exotic homeomorphisms in Theorem~\ref{C2.T1}. In Section~\ref{C3} we show how the $t$-invariant on simply connected $7$-manifolds relates to the $s$-invariants and their generalisations by Hepworth in~\cite{Hep}.  Section~\ref{D} gives our proof of the Feder-Gitler criteria as well as several explicit computations of the $t$-invariant for certain bundles over~$S^7$, $S^{11}$ and~$S^{15}$.

{\bf Acknowledgements:}  It is a pleasure to thank Wolfgang Ziller for raising questions which stimulated this paper and Nitu Kitchloo for his interest in the problem.  We would also like to thank Matthias Kreck for helpful comments and Ian Hambleton for a tea-time discussion in which Theorem~\ref{B3.T1} was formulated.

\section{Secondary Invariants}\label{B}
In this section we define the $t$-invariant which is an invariant of pairs~$(M, E)$ where~$M$ is a closed spin $(4k-1)$-manifold and~$E$ is a quaternionic line bundle over~$M$.  We will give both an extrinsic and an intrinsic definition
of the $t$-invariant.  Whereas the extrinsic definition using zero bordisms
is easier to handle in many cases,
the intrinsic definition using $\eta$-invariants
is finer for manifolds that are only rationally zero bordant and may sometimes be computed when $0$-bordisms cannot be found.  We will state and prove properties of the $t$-invariant in both settings when we are able to because the extrinsic proofs are easier in most cases.  

The section is organised as follows.  Section~\ref{A1} briefly reviews quaternionic line bundles, Section~\ref{B1} defines the $t$-invariant and identifies some of its basic properties.  Section~\ref{A3} introduces the notion of a {\em quaternionic divisor} used in the proof of Theorem~\ref{B3.T1} which occupies most of Section~\ref{B3}.  Finally in Section~\ref{B4} we briefly consider dimension~$3$ and relate an invariant of Deloup and Massuyeau~\cite{DM} to the Atiyah-Patodi-Singer $\xi$-invariant which is precisely analogous to the $t$-invariant.

\subsection{Quaternionic Line Bundles}\label{A1}

Let~$\H, \H^\times$ and~$S^3$ denote respectively the quaternions, the non-zero quaternions and the unit quaternions.  A quaternionic line bundle, $E \to X$ or simply~$E$,  over a space~$X$ is a complex rank~$2$ vector bundle~$V \to X$ together with a reduction of its structure group to~$\H^\times$.
An isomorphism of quaternionic line bundles is a vector bundle isomorphism respecting the quaternionic structures.
As an example, consider the tautological bundle
\[ H \longrightarrow \H P^\infty = BS^3\;.\]

Since the structure group~$\H^\times$ of a quaternionic line bundle can always be reduced to~$S^3$,
and this reduction is unique up to contractible choice, quaternionic line bundles are classified by
the 
homotopy classes of maps to~$BS^3$.
So
\[ \Bun(X) = [X,BS^3] \] 
describes the set of isomorphism classes of quaternionic line bundles over a space~$X$. 
By abuse of notation, we will sometimes write~$E\in\Bun(X)$ to say that~$E$
is a quaternionic line bundle over~$X$.

%
%

The integral cohomology of~$BS^3$ is generated by the universal
second Chern class~$c_2(H)\in H^4(\H P^\infty)$, and so to each bundle~$E\to X$ with classifying map~$f \colon X\to BS^3$ we may associate the characteristic class
$$c_2(E) :=f^*c_2(H) \in H^4(X) \; .$$
In this way we obtain a map
\[ c_2 \colon \Bun(X) \longrightarrow H^4(X)\;,\qquad E \longmapsto c_2(E) \; .\]

Note that the theory of quaternionic bundles is not quite as straight forward as in the complex case.  First, the tensor product of two quaternionic line bundles is merely a four-dimensional real vector bundle.  In fact, every four-dimensional vector bundle that is orientable and spin
arises this way.  In particular, there is no canonical group structure on quaternionic line bundles.

Next, quaternionic line bundles are not classified by their second Chern class.
The reason is that~$BS^3$ is not an Eilenberg-MacLane space.
For example, the quaternionic line bundles on~$S^7$ are classified by
	$$\pi_7(BS^3)\cong\pi_6(S^3)\cong\Z/12\Z\;,$$
but of course~$H^4(S^7)=0$.
On the other hand, for a general space~$X$, not all the elements in~$H^4(X)$ are realised
as~$c_2(E)$ for a quaternionic line bundle~$E\to X$: for example, as is well known and explained in Section~\ref{D}, if~$X = \H P^2$, then the image of~$c_2$ in the group $H^4(\H P^2) \cong \Z$ is precisely all those classes~$c \cdot c_2(H)$ such that~$c(c-1)/2$ is congruent to~$0$ mod~$12$.  

\subsection{Modified Kreck-Stolz Invariants}\label{B1}
Let~$M$ be a closed spin manifold of dimension~$4k-1$ and let~$E \to M$ be a complex vector bundle over~$M$.  As observed by Kreck and Stolz, a source of $\Q/\Z$-valued invariants of the pair~$(M, E)$ is the following index theorem.

\begin{Theorem}[Atiyah-Singer~{\cite[III, Theorem~5.3]{AS}}]\label{B1.T1}
Let~$X^{4k}$ be a closed, smooth spin manifold and let~$E$ be a complex vector bundle over~$X$.  Then the index of~$D_X^{E}$, the Dirac operator of~$X$ twisted by~$E$, can be computed by the following equation:
\begin{equation} \label{B1.Eq1}
\ind\bigl( D_X^{E} \bigr) = \bigl( \widehat A(TX) \ch(E) \bigr) [X] \quad \in \quad \Z \;,
\end{equation}
where~$\widehat A(TX)$ and~${\rm ch}(E)  \in H^*(X)$ denote respectively the~$\widehat A$-genus of~$TX$ and the Chern character of~$E$ and~$[X] \in H_{4k}(X)$ is the fundamental class of~$X$.
\end{Theorem}

Assume that~$(M, E)$ as above bounds a spin manifold and complex vector bundle~$(W, \overline E)$\footnote{Note that throughout $\overline E$ always denotes a bundle over a manifold $W$ with boundary such that $E$ is the restriction of $\overline E$ to the boundary $\del W$.  In particular $\overline E$ {\em is not} the complex conjugate of $E$.} and that the characteristic class~$\widehat A(TX) \ch(\overline E)$ has a natural interpretation in~$H^{4k}(W,M;\Q)$. Then evaluation on the relative fundamental cycle~$[W,M]$
gives a rational number which, by Theorem~\ref{B1.T1} is an invariant of~$(M, E)$ mod~$\Z$. The Atiyah-Patodi-Singer index theorem~\ref{B2.T1} below can also be used to identify this invariant as an intrinsic invariant of~$(M, E)$.  In the remainder of this section, we carry out this program for quaternionic line bundles.


Recall that a {\em real structure\/}~$s$ ({\em quaternionic structure\/}~$j$)
on a complex vector bundle~$E\to X$
is an anti-linear automorphism with~$s^2=1$ ($j^2=-1$).
A quaternionic vector bundle~$E\to X$ can be regarded as a complex vector
bundle with complex structure~$i$ and quaternionic structure~$j$.
We can always choose a quaternionic-Hermitian metric on~$E$
and a connection such that the quaternionic structure is parallel.

If $X^{4k}$ is a spin manifold and $k$ is even then the complex spinor bundle~$\Sigma X$ of~$X$
carries a real structure~$s$
that is parallel with respect to the Levi-Civita connection.
Then~$s\otimes j$ is a quaternionic structure on~$\Sigma X\otimes_\C E$
that commutes with the Dirac operator.
This implies that~$\H$ acts on all eigenspaces of~$D_X^E$,
in particular
\begin{equation}\label{B1.1}
  \ind\bigl(D_X^E\bigr)\in 2\Z\qquad\text{if~$\dim X\equiv 0$ mod~$8$.}
\end{equation}

Because~$c_2$ generates the ring of characteristic classes of quaternionic line
bundles,
there exists a universal characteristic class~$\ch'$ of quaternionic line bundles
such that
\begin{equation}\label{B1.2}
	2-\ch(E)=c_2(E)\,\ch'(E)\;. 
\end{equation}
In fact, one computes that
\begin{equation} \label{B1.2a}
\ch'(E) = 1 - \frac{1}{12}c_2(E) + \dots \;\;.
\end{equation}

Now suppose that~$M$ bounds a compact spin manifold~$W$ such that~$E$ is the restriction of a bundle~$\overline E\to W$.
In other words,
\begin{equation}\label{B1.3}
  [M,E]=0\quad\in\quad\Omega^{\Spin}_*(BS^3)\;.
\end{equation}
Assume moreover that~$H^3(M;\Q)=H^4(M;\Q)=0$.  Then because of the exact sequence
	$$\begin{CD}H^3(M;\Q)@>>>H^4(W,M;\Q)@>>>H^4(W;\Q)
	@>>>H^4(M;\Q)\;,\end{CD}$$
the class~$c_2(\overline E)\in H^4(W;\Z)$
has a unique lift~$\bar c_2(\overline E)\in H^4(W,M;\Q)$.

As a preliminary definition of the $t$-invariant,
let us consider the quantity
\begin{equation}\label{B1.D1} 
  \tau_M(E):=
  \frac{-1}{a_{k+1}}\bigl(\widehat A(TW)\,\ch'(\overline E)\,\bar c_2(\overline E)\bigr)
       [W, M]
       \quad\in\quad\Q/\Z\;,
\end{equation}
where~$a_j$ is~$1$ is~$j$ even and~$2$ if~$j$ is odd.
Theorem~\ref{B1.T1} and property~\eqref{B1.1} above ensure that~$\tau_M(E)$
is independent 
of the choice of the pair~$(W,\overline E)$.

We now want to define~$\tau_M(E)$ intrinsically on~$M$,
so that we can drop condition~\eqref{B1.3}.
Also, given a particular~$M$
and a particular quaternionic line bundle~$E\to M$,
we can sometimes use the numeric value of~$t(E)$
to derive information about~$M$ and~$E$.
Hence, we want to compute~$t(E)$ before we know how to construct~$W$
and how to extend~$E$ to~$W$.

We equip~$TM$ with a Riemannian metric~$g$.
It gives rise to a Levi-Civita connection~$\nabla^{TM}$ on~$TM$.
Let~$D$ be the untwisted spin Dirac operator of~$(M,g)$
acting on the complex spinor bundle~$\Sigma M\to M$.
We also equip the quaternionic line bundle~$E$ with a quaternionic Hermitean
metric~$h^E$ and a compatible connection~$\nabla^E$.
As above, $D_M^E$ denotes the spin Dirac operator twisted by~$(E,\nabla^E,h^E)$.
Let~$\eta(D_M)$ and~$\eta(D_M^E)$ denote the Atiyah-Patodi-Singer
$\eta$-invariants and put~$h(D_M)=\dim\ker D_M$ and~$h(D_M^E)=\dim\ker D_M^E$.

Let~$(\Omega^\bullet(M),d)$ denote the de Rham complex of~$M$.
We consider the characteristic Chern-Weil forms
	$$\widehat A\bigl(TM,\nabla^{TM}\bigr)\;,\quad
	c_2(E,\nabla^E)\;,\quad\text{and}\quad
	\ch(E,\nabla^E)\quad\in\quad\ker d\subset\Omega^\bullet(M)$$
that represent the corresponding characteristic classes in de Rham cohomology.
As in~\eqref{B1.2},
there is a characteristic Chern-Weil form~$\ch'(E,\nabla^E)$ such that
\begin{equation}\label{B2.2} 
  2-\ch(E,\nabla^E)=c_2(E,\nabla^E)\,\ch'(E,\nabla^E)\;.
\end{equation}

If~$H^3(M;\Q)=H^4(M;\Q)=0$, the de Rham cohomology of~$M$ vanishes in
these degrees.
Hence we find~$\hat c_2(E,\nabla^E)\in\Omega^3(M)$ such that
\begin{equation}\label{B2.3}
  d\hat c_2(E,\nabla^E)=c_2(E,\nabla^E)\;,
\end{equation}
and~$\hat c_2(E,\nabla^E)$ is uniquely determined modulo exact forms.

From the variation formulas for $\eta$-invariants and characteristic forms,
it is easy to see that the expression
\begin{multline}\label{B2.4} 
  t_M(E) :=\frac1{a_{k+1}}\,\biggl(\frac{\eta+h}2(D_M^E)
	-(\eta+h)(D_M)\\
	+\int_M\widehat A\bigl(TM,\nabla^{TM}\bigr)
		\,(\hat c_2\,\ch')\bigl(E,\nabla^E\bigr)\biggr)
  \quad\in\quad\R/\Z
\end{multline}
does not depend on the choice of geometric data.
Here we use that~$\frac{\eta+h}2(D_M^E)$ jumps by even integers
as the geometry varies
if~$\Sigma M\otimes_\C E$ carries a quaternionic structure,
see~\eqref{B1.1} above.

In other words, $t_M(E)$ depends neither on the metric~$g$
on~$M$ nor on the connection~$\nabla^E$ with parallel metric~$h^E$ on~$E$.
Alternatively,
regard two sets of data on~$M$ and~$E$.
Because~$\del(M\times[0,1])=M\sqcup(-M)$,
we can extend these data to~$W=M\times[0,1]$ and use Theorem~\ref{B2.T1}
to see that~$t_M(E)$ is indeed well-defined.
The computation is similar as in the proof of Proposition~\ref{B2.P1} below.

Now assume that~$(M,g)$ bounds a Riemannian spin manifold~$(W,\bar g)$
and that~$(E,\nabla^E,h^E)$ extends to~$(\overline E,\nabla^{\overline E},h^{\overline E})\to W$.
Then the pair
\begin{equation}\label{B2.4a}
  \bar c_2\bigl(\overline E,\nabla^{\overline E}\bigr)
  =\bigl(c_2\bigl(\overline E,\nabla^{\overline E}\bigr),
	\hat c_2\bigl(E,\nabla^E\bigr)\bigr)
  \quad\in\quad\Omega^4(W)\oplus\Omega^3(M)
\end{equation}
is closed in the mapping cone
of the pullback~$\Omega^\bullet(W)\to\Omega^\bullet(M)$.
It represents a lift~$\bar c_2(\overline E)$ of~$c_2(\overline E)$ to~$H^4(W,M;\R)$.
By the variation formulas for Chern-Simons classes,
$\bar c_2(\overline E)$ does not depend on~$\nabla^E$ and~$h^E$.
Note that closed forms on~$W$ act on pairs as above by multiplication
in both factors,
and that evaluation on the fundamental cycle~$[W,M]$ is given by
\begin{equation}\label{B2.4b}
  (\alpha,\beta)[W,M]=\int_W\alpha-\int_M\beta\;.
\end{equation}

Let us assume that~$(W,\bar g)$ and~$(\overline E,\nabla^{\overline E},h^{\overline E})$
are of product type near the boundary.
In this case,
the Atiyah-Patodi-Singer boundary conditions
for the spin Dirac operator~$D^{\overline E}_W$ twisted
by~$(\overline E,\nabla^{\overline E},h^{\overline E})$ are defined,
and we can use the following generalisation of Theorem~\ref{B1.T1}.

\begin{Theorem}[Atiyah-Patodi-Singer~{\cite[I, (4.3)]{APS}}]\label{B2.T1}
  Let~$(M,g)=\del(W,\bar g)$ as above
  and let~$(\overline E,\nabla^{\overline E},h^{\overline E})$ denote a complex vector bundle
  with Hermitian connection and parallel metric over~$W$.
  Then
	$$\ind\bigl(D^{\overline E}_W\bigr)
	=\int_W\widehat A\bigl(W,\nabla^W\bigr)
		\,\ch\bigl(\overline E,\nabla^{\overline E}\bigr)
	-\frac{\eta+h}2\bigl(D_M^E\bigr)
		\quad \in\quad\Z\;.$$
\end{Theorem}

Again,
if the Dirac operator commutes with a quaternionic structure
as in~\eqref{B1.1}, then the index above is even.

\begin{Proposition}\label{B2.P1}
  Let~$(E,\nabla^E,h^E)$ be a quaternionic line bundle with quaternionic
  connection and parallel quaternionic-Hermitian metric over~$(M,g)$.
  \begin{enumerate}
  \item\label{B2.P1.1} Assume that~$(M,E)$ bounds~$(W,\overline E)$ as above.
    Then
	$$\tau_M(E)=t_M(E)\;.$$
  \item\label{B2.P1.2} The invariant~$t_M(E)$ is always rational.
  \end{enumerate}
\end{Proposition}

\begin{proof}
  For~\eqref{B2.P1.1},
  we use the Atiyah-Patodi-Singer index theorem as in~\cite{Do1}
  and~\cite{KS} together with~\eqref{B1.1}, \eqref{B2.D1}, \eqref{B2.4a}
  and~\eqref{B2.4b}.
  We find that
  \begin{multline*} 
    -\bigl(\widehat A(TW)\,\bar c_2(\overline E)\,\ch'(\overline E)\bigr)[W,M]\\
    \begin{aligned}
    &=-\int_W\widehat A\bigl(TW,\nabla^{TW}\bigr)
		\,(2-\ch)\bigl(\overline E,\nabla^{\overline E}\bigr)\\
    &\qquad
	+\int_M\widehat A\bigl(TM,\nabla^{TM}\bigr)
		\,(\hat c_2\,\ch')\bigl(E,\nabla^E\bigr)\\
    &\equiv\frac{\eta+h}2\bigl(D_M^E\bigr)-(\eta+h)\bigl(D_M\bigr)\\
    &\qquad
	+\int_M\widehat A\bigl(TM,\nabla^{TM}\bigr)
		\,(\hat c_2\,\ch')\bigl(E,\nabla^E\bigr)
	\quad\mod a_{k+1}\Z\;.
    \end{aligned}
  \end{multline*}

By~\cite{ABP} the Spin bordism group~$\Omega_{4j-1}^{\Spin}$ is a finite $\Z/2$-vector space for~$j > 0$.  From the Atiyah-Hirzebruch spectral sequence computing~$\Omega^{\Spin}_{4k-1}(BS^3)$, the fact that~$H^*(BS^3; \Z) = \Z[c_2]$ where~$c_2$ has degree~$4$ and that each~$\H P^k \subset BS^3$ is a spin manifold we see that~$\Omega^{\Spin}_{4k-1}(BS^3)$ is also a finite $\Z/2$-vector space.  Hence for any $S^3$-bundle~$(M, E)$, there is~$n\in\{1,2\}$ such that~$(M, E)^{\sqcup n}$ (or equivalently~$(M, E)^{\conn n}$) bounds a spin manifold~$W$
  with a quaternionic line bundle~$\overline E$ that extends~$E^{\sqcup n}$.
  By~\eqref{B2.P1.1},
  \begin{equation*}
    n\,t_M(E)=\tau_{M^{\sqcup n}}(E^{\sqcup n})\quad\in\quad\Q/\Z\;.
    \qedhere
  \end{equation*}
\end{proof}

The argument above also shows that in general,
the extrinsic definition~\eqref{B1.D1}
gives the value of~$t_M(E)$ only up to multiples of~$\frac 1n$ in~$\Q/\Z$.

We finally give the definition of the $t$-invariant.

\begin{Definition}\label{B2.D1}
  Let~$M$ be a closed smooth spin $(4k-1)$-manifold such that~$H^3(M;\Q)=H^4(M;\Q)=0$, then we use  equation~\eqref{B2.4} above to define the $t$-{\em in\-var\-i\-ant\/} of~$M$ as the function
  $$t_M\colon\Bun(M) \longrightarrow \Q/\Z, \quad E \joinrel\mapstochar\longrightarrow t_M(E).$$
  We shall say that~$M$ and~$N$ have {\em isomorphic $t$-invariants} if there is a group isomorphism~$A \colon H^4(M) \to H^4(N)$ and a set bijection~$B \colon \Bun(M) \to \Bun(N)$ which are compatible with the second Chern class and the $t$-invariant; i.e. 
\[ \forall E \in \Bun(M), \quad c_2(BE) = Ac_2(E) \text{~and~} t_N(BE) = t_M(E) \;.\]
\end{Definition}
We conclude this subsection by recording some basic facts about the $t$-invariant.  
%

\begin{Proposition} \label{B3.P1}
The function~$t_M \colon \Bun(M) \to \Q/\Z$ has the following properties:
\begin{enumerate}
\item \label{B3.P1.1} Triviality: for the trivial bundle $\eps : = M \times \H$, we have ~$t_M(\eps) = 0$.
\item \label{B3.P1.2} Additivity with respect to connected sum of bundles and manifolds:
\[ t_{M_0 \conn M_1}(E_0 \conn E_1) = t_{M_0}(E_0) + t_{M_1}( E_1) \in \Q/\Z.\]
\item \label{B3.P1.3} Naturality under almost-diffeomorphisms: if~$\Sigma$ is a homotopy sphere and~$f\colon N \conn \Sigma \cong M$ is a diffeomorphism then
\[ t_N\circ f^* = t_M \] 
where we identify $N$ and $N \sharp \Sigma$ as spaces so that $\Bun(N \sharp \Sigma) = \Bun(N)$.
\end{enumerate}
\end{Proposition}

\begin{Remark}
  \begin{enumerate}
  \item Properties~\eqref{B3.P1.1}--\eqref{B3.P1.3} hold
    for any generalised Kreck-Stolz invariants as we now explain.
    Let~$M$ a closed odd-dimensional manifold, equipped with a geometric
    Dirac operator~$D$, and let~$P$ denote the corresponding
    local index class.
    For the $t$-invariant, we take the spin Dirac operator
    and the $\widehat A$-class,
    but one could also consider the signature operator and the $L$-class
    or something similar.
    Then let~$E\to M$ be a vector bundle of a specific type.
    If there exists a natural class~$\alpha\in\Omega^\bullet(M)/\im d$
    such that~$d\alpha=\ch(E)-\rk E$,
    then one can consider the invariant
	$$\frac{\eta+h}2(D^E)-\frac{\eta+h}2(D)\,\rk(E)
		-\int_MP\bigl(TM,\nabla^{TM}\bigr)\,\alpha\in\R/a\Z\;,$$
    where~$a \in \Z$ depends on the index problem.  Note that the Atiyah-Patodi-Singer $\rho$- and $\xi$-invariants~\cite[II]{APS} are special cases, where~$(E,\nabla^E)$ is flat,
so one can take~$\alpha=0$.
    The proof below will work equally well for all invariants of this type.
  \item 
    We shall see later in Theorem~\ref{C2.T1} that the $t$-invariant
    is not in general preserved by homeomorphisms.
  \end{enumerate}
\end{Remark}

\begin{proof}[Proof of Proposition~\ref{B3.P1}]
For the trivial bundle~$\eps$, one can choose~$\hat c_2=0$, so~\eqref{B2.4}
gives~$t_M(\eps)=0$, and~\eqref{B3.P1.1} follows.

For~\eqref{B3.P1.2},
we choose points~$p_i\in M_i$ and assume that~$(E_i,\nabla^{E_i})$ is isomorphic
to a trivial bundle over a small disk~$U_i$ around~$p_i$ in~$M_i$.
Let
	$$W=(M_0\times[0,1])\bconn(M_1\times[0,1])$$
denote the boundary connected sum taken at the points~$(p_0,1)$ and~$(p_1,1)$.
Put
	$$M=\partial W=(M_0\conn M_1)\sqcup(-M_0)\sqcup(-M_1)\;.$$

The boundary connected sum of the pullbacks of~$E_i$ with respect to
the given trivialisations gives a bundle~$\overline E\to W$.
The pullback connections are trivial on~$U_i\times[0,1]$,
so we get a connection~$\nabla^{\overline E}$ by gluing.

We note that~$d\hat c^2(E_i,\nabla^{E_i})=c^2(E_i,\nabla^{E_i})$
vanishes on~$U_i$ by assumption.
Since~$U_i$ is contractible,
there are representatives~$\hat c_2(E_i,\nabla^{E_i})\in\Omega^3(M_i)$
that vanish on~$U_0$ and~$U_1$ as well.
By pullback and gluing,
we obtain a
global form~$\hat c_2\bigl(\overline E,\nabla^{\overline E}\bigr)\in\Omega^3(W)$ with
	$$d\hat c_2\bigl(\overline E,\nabla^{\overline E}\bigr)
	=c_2\bigl(\overline E,\nabla^{\overline E}\bigr)\;.$$
In particular, the lift~$\bar c_2(\overline E,\nabla^{\overline E})$ is exact
in the mapping cone of the pullback~$\Omega^\bullet(W)\to\Omega^\bullet(M)$.

From the extrinsic description~\eqref{B1.D1} 
and Stokes' theorem,
we obtain
\begin{multline*} 
  t_{M_0 \conn M_1}(E_0 \conn E_1) - t_{M_0}(E_0) - t_{M_1}( E_1)
  =\tau_M(E)\\
  =-\int_W\widehat A\bigl(TW,\nabla^{TW}\bigr)
		\,(d\hat c_2\,\ch')\bigl(\overline E,\nabla^{\overline E}\bigr)\\
	+\int_{\partial W}\widehat A\bigl(TM,\nabla^{TM}\bigr)
		\,(\hat c_2\,\ch')\bigl(E,\nabla^E\bigr)
  =0\;.
\end{multline*}

Naturality in~\eqref{B3.P1.3} is clear for diffeomorphisms
and follows from additivity for almost-diffeomorphisms:
add the trivial bundle~$(\Sigma,  \eps)$ to~$(M,E)$.
\end{proof}

\subsection{Divisors}\label{A3}
In this subsection we develop the concept of the {\em divisor}~$(Y, \nu)$
of a quaternionic line bundle~$E$ in analogy with divisors of
complex line bundles.
The main result is Proposition~\ref{B3.P2} which explains
how the $\eta$-invariants involved in the intrinsic definition of~$t_M$
localise near a divisor.

Complex line bundles on algebraic varieties are characterised by divisors,
that is, by subvarieties of codimension~$1$.
A similar construction still works for smooth manifolds.
Let~$L\to M$ be a complex line bundle, and let~$s\in\Gamma(L)$ be a section
that intersects the zero section~$M\subset L$ transversally.
Let~$\nu\to Y$ be the normal bundle of the zero set~$Y=s^{-1}(0)$ of~$s$,
then~$\nu$ inherits a complex structure,
which is equivalent to the choice of an orientation.
Thus a divisor for a complex line bundle is a closed smooth submanifold of
codimension~$2$ with a normal orientation.
Two divisors give the same smooth complex line bundle iff they represent
the same class in~$H^2(M)$.

If~$E\to M$ is a quaternionic line bundle,
we choose a transversal section~$s\in\Gamma(E)$ as above
and put~$Y=s^{-1}(0)$.
Then~$ds|_Y\colon\nu\to E|_Y$ is an isomorphism of real vector bundles,
so~$\nu$ inherits a quaternionic structure.
In particular, the normal bundle is oriented,
and hence defines a class~$c_2(E)\in H^4(M)$.

\begin{Definition}\label{A3.D1}
  Let~$M$ be a compact smooth manifold.
  A {\em quaternionic divisor\/} in~$M$ is a compact codimension-$4$ submanifold~$Y \subset M$
  together with a quaternionic structure on its normal bundle.
 We 
 assume that~$Y$ meets~$\partial M$ transversally in~$\partial Y$.  If~$M$ is closed, two quaternionic divisors in~$M\times\{0\}$, $M\times\{1\}$
  are {\em equivalent\/} if they extend to a quaternionic divisor
  in~$M\times[0,1]$.
\end{Definition}

Note that~$\del Y\subset\partial M$, in particular~$\del Y=\emptyset$ if~$M$ is closed.  The equivalence of quaternionic divisors can be defined similarly in the case where~$\del M\ne\emptyset$.

For example, $\H P^k\subset\H P^{k+1}$ with the obvious quaternionic structure on its normal bundle is a divisor for the tautological bundle~$H\to\H P^{k+1}$.  To reconstruct~$E\to M$ from a divisor~$(Y,\nu)$,
let~$\nu\to Y$ be classified by~$\xi\colon Y\to BS^3$.
Because~$M$ is compact,
$\xi$ factors through some~$\H P^k\subset\H P^\infty=BS^3$.
The Thom space of the tautological bundle over~$\H P^k$ can be identified
with~$\H P^{k+1}$.
Hence, the Pontrjagin-Thom construction gives a classifying map for~$E$
as in the following diagram.
$$\begin{CD}
  M@>>>\H P^{k+1}@>>>\H P^\infty\\
  @AAA@AAA@AAA\\
  Y@>\xi>>\H P^k@>>>\H P^\infty
\end{CD}$$
The vertical arrow on the right is given by inserting~$0$ as first coordinate
and shifting all other coordinates one place to the right.
One similarly checks that two quaternionic divisors in~$M$ give rise to the
same quaternionic line bundle iff they are equivalent.

Representing a quaternionic line bundle~$E\to M$ by a
quaternionic divisor~$(Y,\nu)$ in~$M$ allows us to replace the $\eta$-invariants
in~\eqref{B2.4} by the $\eta$-invariant of an untwisted Dirac operator on~$Y$.
However,
we will see in Remark~\ref{B3.R2} below that it is not possible
to express~$t_M(E)$ solely in terms of~$(Y,\nu)$ without referring
to the ambient manifold~$M$.

For motivation, we assume first that~$(M,E)$ bounds.
Hence let~$\overline E\to W$ be an extension of~$E\to M$,
and let~$s\in\Gamma(\overline E)$ be a transversal section of~$\overline E$
such that~$s|_M\in\Gamma(E)$ is transversal as well.
Then we obtain a divisor
	$$(X,\nu) := \bigl(s^{-1}(0),\overline E|_{s^{-1}(0)}\bigr)$$
of~$\overline E$ such that~$Y := X \cap M$ is a divisor for~$M$.

Let~$\eps \to W$ again denote the trivial quaternionic line bundle.
The normal bundle~$\nu\cong\overline E_\R|_X$ carries a natural spin structure
with half spinor bundles~$\Sigma_\nu^+\cong\eps|_X$
and~$\Sigma_\nu^-\cong\overline E|_X$.
This implies that~$\eps-\overline E$ is a $K$-theoretic direct image
of~$\eps|_X$ under the inclusion~$X\into W$ in the sense of~\cite{AH}
and~\cite{BZ2}; in particular 
\begin{equation}\label{A3.1} 
  \widehat A(\overline E_\R|_X)\,\ch'(\overline E|_X)=1\;.
\end{equation}
This equation can also be derived directly using the splitting principle.

The bundles~$\overline E\to M$ and~$\nu\to X$ are naturally
oriented as real vector bundles by their quaternionic structures.
Let~$\Omega^\bullet_{(1)}$ denote the space of $L^1$-forms,
and let~$\psi(\overline E,\nabla^{\overline E})\in\Omega^3_{(1)}(\overline E)$
denote the Mathai-Quillen current described in~\cite{BZ1} with
	$$d\psi\bigl(\overline E,\nabla^{\overline E}\bigr)
	=\pi^*c_2\bigl(\overline E,\nabla^{\overline E}\bigr)-\delta_0
	\quad\in\Omega^4_{(1)}(\overline E)\;,$$
where~$\delta_0$ denotes the distribution of integration over the zero section.
This current can be pulled back to~$W$ by the transversal section~$s$,
and we have
\begin{equation}\label{A3.2}
  d\bigl(s^*\psi(\overline E,\nabla^{\overline E})\bigr)
  =c_2\bigl(\overline E,\nabla^{\overline E}\bigr)-\delta_X
	\quad\in\Omega^4_{(1)}(W)\;,
\end{equation}
where~$\delta_X$ now denotes the current of integration over~$X$.
We conclude that
\begin{equation}\label{A3.3}
  d\bigl((\hat c_2-s^*\psi)(E,\nabla^{E})\bigr)
  =\delta_Y\in\Omega^4_{(1)}(M)\;.
\end{equation}
By our assumptions on~$M$,
this property determines~$(\hat c_2-s^*\psi)(E,\nabla^{E})$ up
to exact currents.

The connection~$\nabla^{\overline E}$ induces a natural Riemannian metric
on the total space of~$\overline E\to W$,
and we assume that~$ds$ induces an isometry of a neighbourhood of~$X$
with a neighbourhood of the zero section in~$\overline E|_X$.
By~\eqref{B1.D1}, \eqref{A3.1} and~\eqref{A3.2},
\begin{align} 
  \begin{split}\label{B3.1}
    \tau_M(E)
    &= \frac{-1}{a_{k+1}}\int_W\widehat A\bigl(TW,\nabla^{TW}\bigr)
		\,(c_2\,\ch')\bigl(\overline E,\nabla^{\overline E}\bigr)\\
    &\qquad
	+\frac1{a_{k+1}}\int_M\widehat A\bigl(TM,\nabla^{TM}\bigr)
		\,(\hat c_2\,\ch')\bigl(E,\nabla^E\bigr)\\
    &=\frac{-1}{a_{k+1}}\int_X\widehat A\bigl(TX,\nabla^{TX}\bigr)\\
    &\qquad
	+\frac1{a_{k+1}}\int_M\widehat A\bigl(TM,\nabla^{TM}\bigr)
		\,\bigl((\hat c_2-s^*\psi)\,\ch'\bigr)\bigl(E,\nabla^E\bigr)\;.
  \end{split}
\end{align}

Since we have assumed that~$M$ and~$W$ are spin,
there are induced spin structure on~$X$ and~$Y=X\cap M$.
We consider the untwisted Dirac operators~$D_X$ and~$D_Y$
with respect to these spin structures.
Note that~$\Sigma X$ and~$\Sigma Y$ carry a quaternionic structure
iff~$\dim X\cong 4$ mod~$8$.
But this is the case iff~$\Sigma W$ and~$\Sigma M$ carry real structures,
and hence~$D_M^E$ respects a quaternionic structure on~$\Sigma M\otimes E$.

\begin{Proposition}\label{B3.P2}
  Let~$(Y,\nu)$ be a divisor of a quaternionic line bundle~$E\to M$
  and assume that a neighbourhood of~$Y$ in~$M$ is isometric to
  a neighbourhood of the zero section of~$\nu$,
  where the metric on the total space of~$\nu$ is induced by
  a quaternionic metric and a compatible connection.
  Let~$\alpha_Y\in\Omega^3_{(1)}(M)$ denote a current with~$d\alpha_Y=\delta_Y$.
  Then
  \begin{equation*} 
    t_M(E)
    =\frac{-1}{a_{k+1}}\,\biggl(\frac{\eta+h}2(D_Y)
	-\int_M\widehat A\bigl(TM,\nabla^{TM}\bigr)\,\alpha_Y\,\ch'(E,\nabla^E)
	\biggr)\;.
  \end{equation*}
\end{Proposition}

\begin{proof}
  If~$(M,E)$ bounds,
  this follows from~\eqref{A3.3}, \eqref{B3.1} and the Atiyah-Patodi-Singer
  index theorem~\ref{B2.T1} applied to the untwisted Dirac operator~$D_X$
  on~$X$.

  If~$(M,E)$ does not bound,
  we use Bismut-Zhang's formula for the $\eta$-invariant of direct
  images in~\cite{BZ2} to derive the result directly from~\eqref{B2.4}.
\end{proof}

\begin{Remark}\label{B3.R1}
  Note that Bismut and Zhang use a current~$\gamma\in\Omega^\bullet_{(1)}(M)$
  in~\cite{BZ2} such that
  \begin{align*} 
    d\gamma
    &=2-\ch\bigl(E,\nabla^E\bigr)
        -\widehat A\bigl(\nu,\nabla^\nu\bigr)^{-1}\,\delta_Y\;. 
  \end{align*}
  By~\eqref{B1.2}, \eqref{A3.1} and~\eqref{A3.2},
  the current~$(s^*\psi\,\ch')(E,\nabla^E)-\gamma$ is closed.
  Because in our special situation,
  the current~$\gamma$ of~\cite{BZ2} is constructed directly in terms of~$s$,
  there exists a natural current on~$E$ whose pullback
  by~$s$ becomes~$\gamma$.
  In particular,
  we can work with universal currents on~$\H P^N$ for~$N$ sufficiently large.
  Since~$\H P^N$ has no cohomology in degrees~$4k-1$,
  we conclude that~$(s^*\psi\,\ch')(E,\nabla^E)-\gamma$ is exact,
  so that we replace~$\gamma$ by~$(s^*\psi\,\ch')(E,\nabla^E)$.
\end{Remark}

\begin{Remark}\label{B3.R2}
  There is a fundamental difference between
  the class~$\hat c_2\in\Omega^3(M)/\im d$
  and the currents~$s^*\psi$ and~$\gamma\in\Omega^3_{(1)}(M)$.
  The latter are natural in~$E$ and~$s$ as we have seen in Remark~\ref{B3.R1}.
  In particular,
  we can choose~$s$ in such a way that~$\norm s=1$
  outside a small neighbourhood~$U$ of~$Y$.
  This allows us to choose~$\nabla^E$ in such a way that~$\nabla^Es=0$
  outside~$U$.
  In other words,
  contributions coming from~$s^*\psi$ or~$\gamma$ localise near~$Y$
  and do not involve any other topological information about~$M$.

  On the other hand,
  the class~$\hat c_2$ is constructed using the assumption
  that~$H^3(M;\Q)=H^4(M;\Q)=0$.
  In particular,
  the class~$\hat c_2$ is not natural in~$E$,
  but depends on the topology of~$M$.
  This phenomenon will be exhibited again in Remark~\ref{B3.R3} below.
\end{Remark}

\subsection{The relationship with the Adams \texorpdfstring{$e$}{e}-invariant}\label{B3}
Following the ideas above,
we can identify certain cases where the $t$-invariant can be related
to the Adams $e$-invariant.

\begin{Proposition}\label{B3.P3}
  Let~$M$ be a stably framed $(4k-1)$-manifold with~$k \geq 3$ and~$H^3(M;\Q)=H^4(M;\Q) = 0$.
  Assume that~$E\to M$ is a quaternionic line bundle
  with a classifying map factoring
  through~$\xi\colon M\to S^4\cong\H P^1\subset\H P^\infty$ and that~$x_0\in S^4$ is a regular value of~$\xi$.  If~$Y=\xi^{-1}(x_0)$ then~$Y$ inherits a stable framing from~$M$ and~$\xi$,
  and we have
  \begin{equation*}
    t_M(E)=-e(Y)\;.
  \end{equation*}
\end{Proposition}

\begin{proof}
  By assumption,
  the manifold~$M$ is framed.
  The framing in particular selects a Chern-Weil
  type form~$\widetilde A=\widetilde A(TM\oplus\R^N,\nabla^0,\nabla^{TM\oplus\R^N})$
  such that
  \begin{equation*}
    \widehat A\bigl(TM,\nabla^{TM}\bigr)
    =1+d\widetilde A\;.
  \end{equation*}
  We may assume that~$Y$ is totally geodesic in~$M$.
  Because the normal bundle to~$Y$ is trivialised by~$d\xi$,
  we also have
  \begin{equation*}
    \widehat A\bigl(TY,\nabla^{TY}\bigr)
    =1+d\widetilde A|_Y\;.
  \end{equation*}
  Following~\cite[II, Theorem~4.14]{APS},
  we write the $e$-invariant of~$Y$ with the induced framing as
  \begin{equation*} 
    e(Y)
    =\frac1{a_k}\,\biggl(\frac{\eta+h}2(D_Y)-\int_Y\widetilde A\biggr)\;.
  \end{equation*}

  Consider the tautological bundle~$\tau\to\H P^1$ and assume that~$s_0$
  is a transversal section with~$s_0^{-1}=\{x_0\}$.
  Then~$s=\xi^*s_0$ is a transversal section of~$E$
  with divisor~$Y$.
  Let us pull~$\nabla^E$ back from a connection~$\nabla^\tau$ on~$\H P^1$,
  then~$\ch'(E,\nabla^E)|_Y=1$.
  From Proposition~\ref{B3.P2} and with~$d\alpha_Y=\delta_Y$,
  we have
  \begin{align*} 
    e(Y)+t_M(E)
    &=\frac1{a_k}\,\biggl(\int_M\widehat A\bigl(TM,\nabla^{TM}\bigr)\,\alpha_Y
		\,\ch'\bigl(E,\nabla^E\bigr)-\int_Y\widetilde A\biggr)\\
    &=\frac1{a_k}\,\int_M\Bigl(\alpha_Y\,\ch'\bigl(E,\nabla^E\bigr)
	+d\bigl(\widetilde A\,\alpha_Y\,\ch'(E,\nabla^E)\bigr)\Bigr)
    =0\;,
  \end{align*}
  because~$\ch'(E,\nabla^E)-1\in\Omega^4(M)$ by naturality,
  hence~$\alpha_Y\,\ch'(E,\nabla^E)$ lives in degree~$\le 7<4k-1$.
\end{proof}

The $t$-invariant is clearly unstable from the point of view of $K$-theory,
because any invariant that is stable under addition of trivial vector bundles
is generated by characteristic classes.
On the other hand,
the following result shows that the $t$-invariant on~$S^{4k-1}$ is stable
in the sense of stable homotopy theory.

\begin{Theorem}\label{B3.T1}
Assume that~$k\ge 3$.
For any homotopy sphere~$\Sigma^{4k-1}$ the homomorphism
\[ t_\Sigma \colon \Bun(\Sigma) \cong \pi_{4k-2}(S^3) \longrightarrow \Q/\Z\]
may be identified with the composition
\[ - e \circ S \colon \pi_{4k-2}(S^3) \longrightarrow \pi_{4k-5}^S \longrightarrow \Q/\Z\]
where~$- e \colon \pi_{4k-5}^S \to \Q/\Z$ and~$S \colon \pi_{4k-2} \to \pi_{4k-5}^S$ are respectively the the negative of the Adams $e$-invariant~\cite{A} and the stabilisation homomorphism.
\end{Theorem}

\begin{Corollary} \label{B3.C1}
The $t$-invariant defines an injective homomorphism
\[ t_{S^{11}} \colon \Bun(S^{11}) \cong \pi_{10}(S^3) \longrightarrow \Z/15 \; .\]
\end{Corollary}

\begin{proof}
In this case stabilisation is injective by~\cite[Ch.\,XIII]{T} and the $e$-invariant is injective on the $11$-stem by~\cite[Ex.\,7.17]{A}.
\end{proof}

\begin{proof}[Proof of Theorem~\ref{B3.T1}]
Recall that the boundary map associated to the Hopf fibration
is a homomorphism~$\del\colon\pi_{j+1}(S^4) \to \pi_{j}(S^3)$.
Because the fibre~$S^3$ is contractible in the total space~$S^7$,
the long exact homotopy sequence implies that~$\del$ is surjective.

The Hopf fibration extends to~$\H P^k$ for all~$k$.
From the diagram
\begin{equation*}
  \begin{CD}
    \pi_{j+1}(S^7)@>>>\pi_{j+1}(S^4)@>\del>>\pi_j(S^3)\\
    @VVV@VVV@|\\
    \pi_{j+1}(ES^3)@>>>\pi_{j+1}(BS^3)@>\del>\sim>\pi_j(S^3)
  \end{CD}
\end{equation*}
it follows that the homomorphism induced by the standard inclusion
\[\pi_{j+1}(S^4) \longrightarrow \pi_{j+1}(BS^3) \cong \pi_{j}(S^3)\]
is onto.  
So without loss of generality,
the quaternionic line bundle~$E\to S^{4k-1}$
is classified by a map~$\xi\colon S^{4k-1} \to S^4\subset BS^3$.
Note that~$E \cong \xi^*H$ where~$H$ is the tautological bundle over~$S^4$.

We give~$S^{4k-1}$ the standard framing,
and consider the class
	$$[\xi] \in\pi_{4k-5}^S \cong \Omega^{\rm fr}_{4k-5}
	\cong \Omega^{\rm fr}_{4k-1}(S^4)\;$$
	where~$\Omega^{\rm fr}_*$ denotes the framed bordism and we have applied the Pontrjagin-Thom isomorphism.  By the Pontrjagin-Thom construction,
this class~$[\xi]$ is represented by the quaternionic divisor~$Y=\xi^{-1}(x_0)$
for some regular value~$x_0\in S^4$ of~$\xi$.
In particular,
the manifold~$Y$ inherits a framing from~$\xi$.
Then by Proposition~\ref{B3.P3},
we have
	$$t_{S^{4k-1}}(E)=-e(Y)=-e([\xi])\;.\qedhere$$
\end{proof}

\begin{Remark}\label{B3.R3}
  Theorem~\ref{B3.T1} is void for~$k=0$ and~$1$ ,
  and it does not hold for~$k=2$.
  For the case~$k=2$ consider the sequence of maps
	$$\begin{CD}S^7@>H>> S^4@>F_c>>S^4@>>>\H P^\infty\end{CD}$$
  where the first map is the Hopf fibration~$H$ 
  and~$F_c$ is a self-map of degree~$c\in\Z$.
  Let~$E_c\to S^7$ be the pullback of the tautological bundle~$H_1 \to \H P^1$,
  then~$t_{S^7}(E_c)=\frac{c(c-1)}{24}$ by Example~\ref{D1.K2}.

  On the other hand, let~$x_0 \in S^4$ be a regular value of~$F_c$
  and let~$Y_c \subset S^7$ be its pre-image:
  $Y_c$ if a framed $3$-manifold which is a disjoint union
  of fibres of the Hopf fibration.
  It is well known that the map~$H$ represents a generator
  of the stable $3$-stem with~$e([H]) = \pm \frac{1}{24}$.  It follows that
  \[ e(Y_c) = \pm \frac{c}{24}\]
  and in particular~$t_{S^7}(E_c) \neq -e(Y_c)$ in general.
 
 The following observations view the difference between the $t$-invariant and the $e$-invariant from the homotopic point of view: there is an isomorphism~$\pi_7(S^4) \cong \Z[H] \oplus S\pi_6(S^3)$ where~$S\pi_6(S^3) \cong \Z/12$ is the stabilised group.  In these co-ordinates the stabilisation map~$\pi_7(S^4) \to \pi_8(S^5) = \pi_3^S$ is isomorphic to
 \[ \Z \oplus \Z/12 \longrightarrow \Z/24 \; , \quad (a, [b]) \longmapsto [a - 2b] \; .\]  
 Moreover, in this basis~$[F_c \circ H] = (c^2, [c(c-1)/2])$.  In particular~$[F_c \circ H]$ stabilises to~$c^2 - c^2 + c = c$ and we have the equation
 \begin{equation} \label{B3.Eq1}
 t_{S^7} = \pm \frac{\bar e(\bar e \pm \frac{1}{24})}{24} \colon \Bun(S^7) \cong \pi_6(S^3) \longrightarrow \Q/\Z
 \end{equation}
where~$\bar e = e \circ S \colon \pi_6(S^3) \to \pi_3^S \to \Q/\Z$.

We sketch a proof that~$e([H]) = \pm \frac{1}{24}$. The induced framing on a fibre of the Hopf fibration differs from the standard framing by the clutching function of the Hopf fibration itself: call this framed manifold~$(S^3, \pi_H)$.  Now apply the bordism definition of the the $e$-invariant: as a spin manifold~$(S^3, \pi_H)$ bounds over~$D^4$ as~$S^3$ has a unique spin structure.  We see that
 \[ e(S^3, \pi_H) = \frac{1}{2}\widehat A\bigl(TD;\pi_H\bigr)[D^4,S^3] = \frac12\widehat A(H)[S^4] = \pm \frac{1}{24} \]
 where we regard~$H$ as the Hopf bundle over~$S^4$
 with~$p(H)$ a generator of~$H^4(S^4)$ and~$\widehat A_1 = -\frac{1}{12} p$.
 Here, $\widehat A(TD;\pi_H)$ denotes the relative $\widehat A$-class with
 respect to the prescribed framing at the boundary.
  
  As pointed out in Remark~\ref{B3.R2},
  there are at least two different possible correction terms
  that construct a differential topological invariant out of the
  $\eta$-invariant of the untwisted Dirac operator on the divisor~$Y$.
  For~$k=2$ different choices of such correction terms lead to different invariants.
\end{Remark}

\subsection{A related example in dimension \texorpdfstring{$3$}{3}} \label{B4}
Let~$(M, \sigma)$ be a compact oriented 3-manifold~$M$ with $\spin^c$-structure.  DeLoup and Massuyeau~\cite[Definition 2.2]{DM} defined an invariant
\[ \phi_{M, \sigma} \colon H_1(M; \Q/\Z) \longrightarrow \Q/\Z \; .\]
In this subsection we give an analytic definition of~$\phi_{M, \sigma}$ using the Atiyah-Patodi-Singer $\xi$-invariant~\cite[II]{APS} which is a precise analogue of the $t$-invariant in dimension~$7$ if~$M$ is a rational homology sphere.


Let~$W$ be a handlebody with~$\partial W=M$.
Then~$W$ carries a unique spin structure,
which induces a spin structure on~$M$.
There is a one-to-one correspondence between  complex line bundles~$L$
and $\spin^c$-structures~$\sigma$ on~$M$,
given by twisting the fixed spin-structure above by~$L$.
The exact sequence of the pair~$(W,M)$ contains
$$0\longrightarrow H^1(M)\longrightarrow H^2(W,M)\longrightarrow H^2(W)\stackrel{}\longrightarrow H^2(M)\longrightarrow 0\;,$$
so~$L$ extends to a complex line bundle ~$\overline L$ on $W$
that induces a $\spin^c$-structure~$\bar \sigma$ on~$W$ extending~$\sigma$.
We also fix a connection~$\nabla^{L}$ on~$\overline L \to W$.
The Chern forms of these $\spin^c$-structures are given by
	$$c(\bar\sigma)=2c_1\bigl(\overline L,\nabla^{\overline L}\bigr)
	\qquad\text{and}\qquad c(\sigma)=c(\bar\sigma)|_M\;.$$

Let~$a\in H^1(M;\Q/\Z)$,
then we construct a flat complex line bundle~$E_a$ over~$M$
with holonomy given by
	$$a\in H^1(M;\Q/\Z)\cong\Hom\bigl(\pi_1(M),\mu^*\bigr)\;,$$
where~$\mu^*\subset\C$ denote the group of roots of unity
and the isomorphism is induced by the isomorphism~$\Q/\Z\to\mu^*$
sending~$q$ to~$e^{2\pi iq}$.
Let~$\widetilde M$ denote the universal covering of~$M$
and let~$\pi_1(M)$ act as the group of deck transformations.
The bundle~$E_a$ is given by
	$$\widetilde M\times_a\C\qquad\text{with}\qquad
	\bigl[x\gamma,e^{-2\pi i\,a(\gamma)}z\bigr]=[x,z]
        \qquad\text{for all }\gamma\in\pi_1(M)\;.$$
        We define~$\nabla^{E_a}$ to be the connection on~$E_a$ induced by the trivial connection
on~$\widetilde M\times\C$.  If~$\beta\colon H^1(M;\Q/\Z)\to H^2(M,\Z)$ denotes the Bockstein homomorphism,
then~$c_1(E)=\beta(a)$.

By the exact sequence above,
there exists a line bundle $\overline E$ on~$W$ that extends~$E$,
and we can choose an arbitrary connection~$\nabla^{}$ on~$\overline E\to W$,
that will not be flat in general,
but restricts to the given flat connection on~$E \to M$.
If~$\gamma=\partial\Sigma$ is a curve in~$M$,
where~$\Sigma\subset W$ is a closed immersed surface,
then
	$$a(\gamma)
	=c_1(\overline E,\nabla^{\overline E})[\Sigma,\partial\Sigma] \mod\Z\;.$$
This shows in particular that the first Chern
class
	$$c_1(\overline E)=\bigl[c_1\bigl(\overline E,\nabla^{\overline E}\bigr)\bigr]
	\in H^2(W,M;\R)$$
lifts the image of~$a$ in~$H^2(W,M;\Q/\Z)$.
This lift is well-defined only up to an integral class.

We compute the Atiyah-Patodi-Singer $\xi$-invariant of~$E$
with respect to the fixed $\spin^c$-structure~$\sigma$ on~$M$
and obtain
\begin{align*}
  \xi_{M,\sigma}(E_a, \nabla^{E_a})
  &=\frac{\eta+h}2\bigl(D_M^{L\otimes E_a}\bigr)-\frac{\eta+h}2\bigl(D_M^{L}\bigr)\\
  &=\int_W\widehat A\bigl(TW,\nabla^{TW}\bigr)
	\,\ch\bigl(\overline L,\nabla^{\overline L}\bigr)
	\,\Bigl(\ch\bigl(\overline E,\nabla^{\overline E}\bigr)-1\Bigr)\\
  &=\int_W\biggl(1+\frac{c(\bar\sigma)}2\biggr)
	\,\biggl(c_1\bigl(\overline E,\nabla^{\overline E}\bigr)
		+\frac12\,c_1\bigl(\overline E,\nabla^{\overline E}\bigr)^2\biggr)\\
  &=\frac{c_1(\overline E)\,\bigl(c(\bar\sigma)+c_1(\overline E)\bigr)}2\,[W,M]
	\quad\in\R/\Z\;.
\end{align*}
Note that the last term is well-defined because~$c_1(\overline E)$ is well defined
as a relative class modulo~$\Z$.
Comparing the above formula with~\cite[Lemma 2.6]{DM} we obtain

\begin{Proposition}
Let~$(M, \sigma)$ be a closed $\spin^c$-manifold and let~$\wp$ denote Poincar\'e duality.  For all~$a \in H^1(M; \Q/\Z)$
	$$\phi_{M,\sigma}\bigl(\pm\wp\,a\bigr)
	=\pm\xi_{M,\sigma}(E_a,\nabla^{E_a})\,.$$
\end{Proposition}

\section{The \texorpdfstring{$t$}{t}-invariant of \texorpdfstring{$7$}{7}-manifolds}\label{C}

 
In this section we investigate the~$t$-invariant in dimension $7$: in~\ref{C1} we show that the $t$-invariant classifies closed smooth $2$-connected rational homology $7$-spheres up to connected sum with homotopy spheres.  In~\ref{C2} we show that the $t$-invariant detects homeomorphisms which are not homotopic to PL homeomorphisms on $2$-connected rational homology $7$-spheres.  In Sections~\ref{C3} we relate the $t$-invariant to the $s$-invariants of Kreck and Stolz~\cite{KS} and their generalisations in Hepworth~\cite{Hep}.

\subsection{Classification results for 2-connected 7-manifolds} \label{C1}
Throughout this subsection~$M$ shall be a closed smooth oriented $2$-connected $7$-manifold.  In addition we assume that~$M$ is a rational homology sphere which is equivalent to assuming that~$\pi_3(M) \cong H_3(M) \cong H^4(M)$ are finite groups.  

We first recall the classification of such $2$-connected rational homology spheres started in~\cite{Wi} and completed in~\cite{Crow}.  We then relate the $t$-invariant from Section~\ref{B} to this classification to obtain a classification theorem for such manifolds~$M$ using~$t_M$ and~$\mu(M)$, the Eells-Kuiper invariant of~$M$ which we recall below.

Recall that~$\Theta_7$ denotes the group of diffeomorphism classes of oriented homotopy 7-spheres and that~$\Theta_7 \cong \Z/28$.  The homotopy 7-spheres can be detected by the Eells-Kuiper invariant~$\ek(\Sigma) \in \Q/\Z$ which by~\cite[\S 6]{EK} defines an injective homomorphism
\[ \ek \colon \Theta_7 \longrightarrow \Q/\Z \; .\]
Moreover the definition of~$\ek$ can be extended to any 2-connected rational homotopy 7-sphere~$M$ to give~$\ek(M) \in \Q/\Z$ with
\begin{equation} \label{C1.Eq1}
\ek(M \conn \Sigma) = \ek(M) + \ek(\Sigma) \in \Q/\Z \; .
\end{equation} 
The definition of~$\mu(M)$ in this case is by now routine: from the analytic point of view it can be found in~\cite{Do1, KS}; using coboundaries the details are in~\cite[2.12]{Crow} where~$\mu$ is called~$s_1$.  With either definition the additivity in~\eqref{C1.Eq1} is clear.

This gives an effective strategy for classifying 2-connected rational homology 7-spheres.  Firstly consider these manifolds up to {\em almost diffeomorphism} where an almost diffeomorphism is a homotopy sphere~$\Sigma$ and a diffeomorphism~$f \colon N \conn \Sigma \cong M$.  If~$M$ and~$N$ are almost diffeomorphic then they are diffeomorphic if and only if~$\mu(M) - \mu(N) = \mu(\Sigma) = 0$.

We turn to consider almost diffeomorphism invariants of~$M$.  Since~$M$ is $2$-connected it possesses a unique equivalence class of spin structures.  An important invariant is the first characteristic class of spin manifolds which generates~$H^4(B\Spin) \cong \Z$.  We choose the generator~$p \in H^4(B\Spin)$ so that that~$2 p = p_1$ where~$p_1$ is the first Pontrjagin class.  The class~$p$ is often called ``half the first Pontrjagin class'' and is sometimes denoted~$\frac{p_1}{2}$.  For any spin manifold~$X$, we have~$p_X=p(TX) \in H^4(X)$.

An important fact about closed smooth $2$-connected $7$-manifolds proven in~\cite[Theorem 4]{Wi} is that every one is the boundary of a handlebody~$W$: i.e.~$W$ is a 3-connected 8-manifold obtained from~$D^8$ by attaching 4-handles, and the boundary of~$W$ is identified with~$M$.
We consider the exact sequence
\begin{equation}\label{C1.0}
    0\longrightarrow  H^4(W, M)\longrightarrow  H^4(W)\stackrel{i^*}\longrightarrow H^4(M)\longrightarrow  0\;,
\end{equation}
which 
allows one to define the following quadratic linking function
\begin{equation}\label{C1.1}
  q_M\colon H^4(M) \longrightarrow \Q/\Z \;, \qquad x
  \longmapsto \frac{1}{2}\bigl((\bar x + p_W) \smile \hat x\bigr)[W,M] \;,
\end{equation}
where~$\bar x \in H^4(W)$ maps to~$x\in H^4(M)$
and~$\hat x \in H^4(W, M; \Q)$ maps to~$\bar x$ under the isomorphism~$H^4(W, M; \Q) \cong H^4(W; \Q)$.  

\begin{Lemma}[c.f.\,{\cite[Lemma 2.51]{Crow}}]
The function~$q_M$ is a well-defined almost diffeomorphism invariant of~$M$.
\end{Lemma}
We shall say that~$q_M$ and~$q_N$ are isomorphic if there is an isomorphism~$A \colon H^4(M) \cong H^4(N)$ such that if~$q_M = q_N \circ A$.  The following is also proven in~\cite{Crow}.
\begin{Theorem}[c.f.\,{\cite[Theorem A]{Crow}}]\label{C1.T1}  
Let~$N$ and~$M$ be 2-connected rational homology 7-spheres.
\begin{enumerate}
\item \label{C1.T1.1} There is an almost diffeomorphism~$f\colon N \conn \Sigma \cong M$ with induced map~$f^* = A \colon H^4(M) \cong H^4(N)$ if and only if~$q_M = q_N \circ A$.
\item \label{C1.T1.2} $M$ is diffeomorphic to~$N$ if and only if~$q_M$ is isomorphic to~$q_N$ and~$\ek(M) = \ek(N)$.
\end{enumerate}
\end{Theorem}

\begin{Remark}
Recall the linking form of~$M$
\[  b_M \colon H^4(M) \times H^4(M) \longrightarrow \Q/\Z \; . \]
which is a non-singular symmetric pairing.
In the notation of~\eqref{C1.1}, we have
\[ b_M(x, y) = (\bar x \smile \hat y)[W,M] = (\hat x \smile\bar y) [W,M] \; .\]
From~\eqref{C1.1} we see that~$q_M$ refines~$b_M$ in the following sense:
\begin{equation}\label{C1.3}
  q_M(x + y) = q_M(x) + q_M(y) + b_M(x, y) \quad \forall \,x, y \in H^4(M) \; .
\end{equation}
We also see that~$q_M$ need not be homogeneous; i.e.~$q_M(x) \neq q_M(-x)$ in general.  However, the homogeneity defect of~$q$ is determined by~$p_M$:
\[ q_M(x) - q_M(-x) = (p_W \smile \hat x)[W,M] = b_M(p_M, x) \; .\]
\end{Remark}


Recall now that~$\Bun(M)$ denotes the set of isomorphism classes of principal $S^3$-bundles over~$M$. 
%
%
We have an action of~$\Bun(S^7) \cong \pi_6(S^3) \cong \Z_{12}$ on~$\Bun(M)$ by
\[ \Bun(M) \times \Bun(S^7) \longrightarrow \Bun(M) \; , \quad (E, F) \longmapsto E \conn F \]
which is defined by cutting and re-gluing a given principal $S^3$-bundle~$E$ over~$M$ along~$S^6 = \del D^7 \subset M$ using the clutching function of~$F$.
Next we relate~$q_M$ to the $t$-invariant~$t_M$ of Definition~\ref{B2.D1}.  

\begin{Theorem}\label{C1.T2}
Let~$M$ be a $2$-connected rational homology 7-sphere.  
\begin{enumerate}
\item \label{C1.T2.1}
For~$M = S^7$, the $t$-invariant defines an injective homomorphism
\[ t_{S^7} \colon \Bun(S^7) \cong \Z/12 \subset \Q/\Z \; .\]
\item \label{C1.T2.2}
The group\/~$\Bun(S^7)$ acts freely on\/~$\Bun(M)$ with
$$\begin{CD}\Bun(M)/\Bun(S^7)@>c_2>\cong>H^4(M)\;,\end{CD}$$
%
%
and for all~$(E, F) \in \Bun(M) \times \Bun(S^7)$ we have
\[ t_M(E \conn F) = t_M(E) + t_{S^7}(F)\; .\]
\item\label{C1.T2.3}
The $t$-invariant determines~$q_M$: for all~$E \in \Bun(M)$
\[ 
q_M(c_2(E))= 12\,t_M(E) \in \Q/\Z \; .\]
\end{enumerate}	
\end{Theorem} 
\noindent
As an immediate consequence of Theorem~\ref{C1.T1}~\eqref{C1.T1.2} and Theorem~\ref{C1.T2}~\eqref{C1.T2.3} we have 

\begin{Corollary} \label{C1.C1}
Let~$N$ and~$M$ be $2$-connected rational homology $7$-spheres.  Then~$N$ is diffeomorphic to~$M$ if and only if~$t_N$ is isomorphic to~$t_M$ and~$\mu(N) = \mu(M)$.
\end{Corollary}

\begin{proof}[Proof of Theorem~\ref{C1.T2}]
For (1) we use Proposition~\ref{S4.T1} below with~$(n, p, k) = (1, 2, 1)$.  One checks that this gives the bundle~$E = p_2^*E_2$ of Example~\ref{D1.K2} which is a bundle over~$S^7$ with~$t_{S^7}(E) = \frac{1}{12}$.  By Proposition~\ref{B3.P1}~\eqref{B3.P1.2},
  $t_{S^7}$ is a homomorphism with~$t_{S^7}(E^{\conn k}) = \frac{k}{12}$.

For part~\eqref{C1.T2.2}, let~$W$ be a handlebody with~$\partial W=M$ as above.
As~$W$ is homotopy equivalent to a wedge of $4$-spheres and~$BS^3$
is three-connected,
we have an isomorphism~$c_2\colon\Bun(W)\to H^4(W)$.
Hence~$c_2\colon\Bun(M)\to H^4(M)$ is onto by~\eqref{C1.0}.
There is a homotopy equivalence~$M \simeq M^\bullet \cup e^7$,
where~$M^\bullet : = M - {\rm int}(D^7)$ is homotopy equivalent
to a degree $3$-Moore space.
Because~$BS^3$ is three-connected,
the map~$c_2\colon\Bun(M^\bullet)\to H^4(M^\bullet)\cong H^4(M)$
is an isomorphism.
By the surjectivity of~$c_2\colon\Bun(M)\to H^4(M)$,
each quaternionic line bundle~$E^\bullet$ on~$M^\bullet$ extends to~$M$,
so~$E^\bullet|_{\partial e^7}$ is trivial
and~$\Bun(S^7)$ acts transitively on the set of possible extensions.

  
  The formula for~$t_M$ given in~\eqref{C1.T2.2} is a special case of the additivity formula of Proposition~\ref{B3.P1}~\eqref{B3.P1.2}.  Together with~\eqref{C1.T2.1}, it proves that the action of~$\Bun(S^7)$ on~$\Bun(M)$ is free.
  

For part~\eqref{C1.T2.3},
let~$\overline E\to W$ be a quaternionic line bundle
with~$c_2(\overline E)=\bar x\in H^4(W)$.
Let~$\hat c_2(\overline E)\in H^4(W,M;\Q)$ be a lift of~$c_2(\overline E)$.
We use equation~\eqref{B1.D1},
the fact that the $\widehat A$-genus begins~$\widehat A = 1 - \frac{p}{12} + \dots$ and~\eqref{B1.2a} to deduce the following formula
\begin{equation}  \label{C1.2} 
 t_M(i^*E) = \frac{1}{24}\bigl(\hat c_2(\overline E) \smile (p_W + c_2(\overline E) \bigr)[W,M] \in \Q/\Z \; .
\end{equation} 
Comparing~\eqref{C1.2} with~\eqref{C1.1} we see that~$q_M(c_2(E)) = 12t_M(E)$.
\end{proof}

%
%

We finish this subsection by recording calculations of~$t_M$ in some examples.  Let $n$ and $p$ be integers with $n \neq 0$ and let~$\pi \colon W_{n, p} \to S^4$ be the disc bundle of a vector bundle over the $4$-sphere with Euler class $e(\pi) = n \bar x$ and with $p_W = p \bar x$: here we fix a generator $\bar x$ of $H^4(S^4)$ and identify $H^4(S^4) = H^4(W_{n, p})$.   Let also~$i \colon M_{n, p} \to W_{n, p}$ be the inclusion of the boundary so that $M_{n, p}$ is the total space of a $3$-sphere bundle over $S^4$.  We remark that these total spaces were classified up to diffeomorphism, homeomorphism and homotopy equivalence in \cite{CE} but using a different notation: the total space denoted $M_{m, n}$ in \cite{CE} is the total space denoted $M_{n, n+2m}$ in this work (see, e.g.,  \cite[Fact 3.1]{CE}).

\begin{Proposition} \label{S4.T1}
With~$i \colon M_{n, p} \to W_{n, p}$ and $\bar x \in H^4(S^4)$ as above and $k$ any integer, let~$\overline E_k \to W_{n, p}$ be the quaternionic line bundle with $c_2(\overline E) = k \bar x$ and let $E_k := i^* \overline E_k$.  Then for $M = M_{n, p}$ we have
	$$t_{M_{}}(i^*E_k) 
	=\frac{k(p + k)}{24n}\quad\in\quad\Q/\Z\;.$$
  Moreover~for~$[k] := i^*(k \bar x) \in\Z/n\Z\cong H^4(M)$, the function~$q_M$ is given by
	$$q_{M_{}} ([k])=\frac{k(p + k)}{2n}\quad\in\quad\Q/\Z\;.$$
\end{Proposition}

\begin{proof}
The intersection form~$(H_4(W_{n, p}), \lambda)$ is isomorphic to~$(\Z, n)$.  We now simply apply the expression for~$t_M$ in~\eqref{C1.2}.
\end{proof}

\begin{Remark}
While it is possible to use the analytic definition of the $t$-invariant to compute~$t_M$ in this case, the computations still require sophisticated technique and are somewhat lengthy compared to the topological definition.
\end{Remark}

%
%
%
%

In the search for new compact Riemannian
manifolds of positive sectional curvature,
Grove, Wilking and Ziller considered two families~$M_{(p_-,q_-),(p_+,q_+)}$
and~$N_{(p_-,q_-),(p_+,q_+)}$ of seven-manifolds of cohomogeneity one
in~\cite{GWZ}.
The first family consists of two-connected manifolds,
while the second is of the type considered by Kreck and Stolz in~\cite{KS}.
By~\cite{GWZ},
the manifolds~$P_n=M_{(1,1),(2n-1,2n+1)}$ are the only members
of the $M$-family that can carry metrics of positive sectional curvature.
It was then proved independently by Dearricott~\cite{Dear} and Grove, Verdiani
and Ziller~\cite{GVZ} that the particular space~$M_{(1,1),(3,5)}$,
which is homeomorphic to the unit tangent bundle of~$S^4$,
carries a metric of positive sectional curvature.
In~\cite{Goe},
one of us determined the diffeomorphism types of the spaces~$P_n$
using Corollary~\ref{C1.C1} above.

Let~$(p_+,q_+)$ and~$(p_-,q_-)$ be two pairs of relative prime
positive odd integers.
According to~\cite{GWZ},
the manifold~$M=M_{(p_-,q_-),(p_+,q_+)}$
has~$H^4(M)\cong\Z/n\Z$ with~$n=\left|p_-^2q_+^2-p_+^2q_-^2\right|/8$.
One can write~$M$ as the total space
of two Seifert fibrations~$\pi_1$, $\pi_2$ over the base~$S^4$.
Both Seifert fibrations are singular
over two disjoint submanifolds of~$S^4$ diffeomorphic to~$\R P^2$,
and the type of the singularity is described by the two pairs~$(p_\pm,q_\pm)$.
Using a generalisation of the adiabatic limit formula for eta-invariants
to Seifert fibrations,
it is possible to compute~$t_M(\pi_i^*E)$ for all quaternionic line
bundles~$E\to S^4$.

\begin{Example} \label{C1.E1}
  Let~$E\to S^4$ be a quaternionic line bundle with~$c_2(E)[S^4]=k\in\Z$.
  By~\cite{Goe},
  	$$t_M(\pi_1^*E)=\frac{k(p_+^2-p_-^2-n+k p_-^2p_+^2)}{24n}\;.$$
  Swapping the roles of~$p_\pm$ and~$q_\pm$ gives the analogous
  formula for~$t_M(\pi_2^*E)$.
\end{Example}

For the manifolds~$P_n$,
the formula above specialises to
	$$t_{P_n}(\pi_1^*E)=\frac{k(k-n)}{24n}\;.$$
By Theorems~\ref{C1.T1} and~\ref{C1.T2} and Proposition~\ref{S4.T1} above,
these spaces $P_n$ are almost diffeomorphic to the manifolds~$M_{n,n}$ of Proposition \ref{S4.T1}:
in other words, to the total spaces of the principal $S^3$-bundles over~$S^4$
with Euler class given by~$n$.
The classification is then completed
by computing the Eells-Kuiper invariant~$\mu(P_n)$.
In particular,
	$$P_n\cong M_{n,n}\conn\Sigma^{\#\frac{n-n^3}6}\;,$$
where~$\Sigma$ is a generator of~$\Theta_7$ with~$\mu(\Sigma)=\frac1{28}$.

For the other spaces in the~$M$ family,
one can to combine the formulas for~$t_M(\pi_i^*E)$ for both fibrations
with equation~\eqref{C1.3} to determine~$q_M$ completely.

\subsection{Detecting exotic homeomorphisms} \label{C2}
Throughout this subsection~$M$ and~$N$ will be closed smooth oriented 2-connected rational homology $7$-spheres and all maps will be orientation preserving.  We call a homeomorphism
\[ h \colon N \cong M\]
{\em exotic} if it is not homotopic to a piecewise linear (PL) homeomorphism.  The main result of this section is that~$h$ is exotic if and only if the induced map~$h^* \colon \Bun(M) \to \Bun(N)$ does not preserve the $t$-invariants of~$M$ and~$N$.

To make a more precise statement we first recall the following consequence of topological surgery from~\cite{KiSi}.  For a homeomorphism~$h \colon N \cong M$ there is an invariant 
\[ \KS(h) \in H^3(M; \Z/2) \; , \]
depending only on the homotopy class of~$h$, such that~$h$ is exotic if and only if~$\KS(h) \neq 0$.

\begin{Theorem} \label{C2.T1}
A homeomorphism~$h: N \to M$ is exotic if and only if
\[t_M \neq t_N \circ h^*.\]
More precisely, for all~$E \in \Bun(M)$,  
\[ \bigl(\KS(h) \smile c_2(E) \bigr) [M]_2  =  t_N(h^*E) - t_M(E) \in \Z/2 \subset \Q/\Z \;, \]
where~$[M]_2$ generates~$H_7(M; \Z/2)$. 
\end{Theorem}

We begin by showing that~$t_M$ is invariant under PL-homeomorphisms.

\begin{Lemma} \label{C2.L1}
If~$h \colon N \cong M$ is a $PL$-homeo\-morphism then~$t_M = t_N \circ h^*$.
\end{Lemma}

\begin{proof}


It follows from smoothing theory, see~\cite[\S 6,\,Theorem 6.1]{Crow}, that there is a homotopy $7$-sphere~$\Sigma$ such that~$h$ is homotopic to a diffeomorphism~$g \colon N \conn \Sigma \to M$.  The lemma now follows by Proposition~\ref{B3.P1}~\eqref{B3.P1.3}.
\end{proof}

\begin{Remark}
Note that~\cite[Theorem 6.1]{Crow} also proves that~$N$ is homeomorphic to~$M$ if and only if~$N$ is $PL$-homeomorphic to~$M$.  Hence the function~$t_M$ is a topological invariant in the weak sense that if~$h \colon N \to M$ is a homeomorphism then there is a homeomorphism, indeed a $PL$-homeomorphism, $g\colon N \to M$ such that~$t_M = t_N \circ g^*$.
\end{Remark}


\begin{proof}[Proof  of Theorem~\ref{C2.T1}]  
We will not explicitly construct exotic homeomorphisms but rather we use surgery theory to show that they exist.  Hence we begin by briefly recalling some essential notions from surgery.  Let~$\Cat = \Top$ or ~$PL$ denote respectively the topological and piecewise linear categories of manifolds.  Recall the $\Cat$-structure set of~$M$, 
\[ \mathcal{S}^\Cat (M) := \{g\colon N \simeq M \, | \,  \text{$N$ a $\Cat$-manifold} \} / \sim \; , \]
which consists of equivalence classes of {\em structures} which are homotopy equivalences~$g \colon N \to M$.  Two structures~$g_0$ and~$g_1$ are equivalent if~$g_1 \circ g_0^{-1}$ is homotopic to a $\Cat$-isomorphism.  The base-point of~$\mathcal{S}^{\Cat}(M)$ is the trivial element~$[{\rm Id} \colon M \to M]$.

The $\Cat$-structure set of~$M$ lies in the $\Cat$-surgery exact sequence (see~\cite[Chapter 10]{Wa}, \cite{KiSi} for more definitions and details):
$$L_8(e)\longrightarrow  \mathcal{S}^{\Cat}(M)\stackrel{\eta^{\Cat}}\longrightarrow [M, G/\Cat]\longrightarrow  L_7(e) \;.$$
Here~$L_7(e) = 0$ and~$L_8(e) \cong \Z$ are the simply connected surgery obstruction groups, $G = {\rm lim}_n {\rm Map}_{\pm 1}(S^n, S^n)$ is the monoid of stable self-equivalences of the $n$-sphere, $\Top$ and~$PL$ are the groups of stable homeomorphisms, respectively PL-homeomorphisms, of Euclidean space and~$\eta^\Cat$ denotes the $\Cat$-normal invariant map.  It is well known that the map~$L_8(e) \to \mathcal{S}^\Cat (M)$ vanishes in both the topological or piecewise linear categories.  So from the $\Cat$-surgery exact sequences we obtain the following commuting square:
\begin{equation} \label{C2.E}
  \begin{CD}
   \mathcal{S}^{PL}(M) @>F>> \mathcal{S}^{\Top}(M)\\
    @VV\eta^{PL}V@VV\eta^{\Top}V\\
    [M, G/PL]@>F_*>>[M, G/\Top]
  \end{CD}
\end{equation}
Here each~$\eta^{\Cat}$ is a bijection, $F$ is the forgetful map and the canonical map~$G/PL \to G/\Top$ induces~$F_*$. By definition, an exotic homeomorphism~$h \colon N \cong M$ defines a non-trivial element of~$\mathcal{S}^{PL}(M)$ which maps to the trivial element of~$\mathcal{S}^{\Top}(M)$.  

\begin{Lemma} \label{C2.L2}
There are bijections~$\mathcal{S}^{PL}(M) \equiv H^4(M)$ and~$\mathcal{S}^{\Top}(M) \equiv H^4(M)$ such that the forgetful map~$F$ corresponds to the map~$\times 2$ in the cohomology Bockstein sequence for the coefficient sequence~$0 \to \Z \to \Z \to \Z/2 \to 0$.  In particular, there is a short exact sequence of abelian groups
\[ 0 \longrightarrow  H^3(M; \Z/2) \longrightarrow  \mathcal{S}^{PL}(M) \stackrel{F}{\longrightarrow } \mathcal{S}^{\Top}(M) \longrightarrow  H^4(M; \Z/2) \; . \]
\end{Lemma}
%
%
\begin{proof}
Using~\eqref{C2.E} if suffices to calculate~$F_* \colon [M, G/PL] \to [M, G/\Top]$ which is routine: there is a homotopy equivalence 
\[ M \simeq M^\bullet \cup_\phi e^7\]
where~$M^\bullet := M - {\rm int}(D^7)$, $\phi$ is the attaching map for the top cell of~$M$ and~$M^\bullet$ is homotopic to the Moore space~$M(H, 3)$ where~$H = H_3(M)$.  We write~$H = \oplus_{i=1}^r C_{k(i)}$ as a sum of cyclic groups and so we have a homotopy equivalence 
\[ M^\bullet \simeq \bigvee_{i=1}^rM(\Z/k(i), 3) \simeq \bigvee_{i=1}^r(S^3_i \cup_{k(i)} e^4_i) \]
where~$k(i)$ denotes a map~$S^3_i \to S^3_i$ of degree~$k(i)$.

It is known that the homotopy groups of~$G/PL$ and~$G/\Top$ vanish in odd dimensions and that~$\pi_4(G/PL) \to \pi_4(G/\Top)$ is isomorphic to~$\times 2 \colon \Z \to \Z$.  We deduce that the homomorphism~$[M^\bullet, G/PL] \to [M^\bullet, G/\Top]$ is isomorphic to~$H^4(M^\bullet) \stackrel{\!\times 2}{\longrightarrow} H^4(M^\bullet)$.

It remains to show that the inclusion~$M^\bullet \to M$ of the 4-skeleton induces isomorphisms~$[M, G/\Cat] \cong [M^\bullet, G/\Cat]$.  Let 
\begin{equation} \label{C2.E1}
 C \colon M^\bullet \longrightarrow  \bigvee_{i=1}^rS^4
\end{equation}
be the map collapsing the $3$-skeleton and observe that the homomorphism
 \[ C^* \colon \oplus_{i=1}^r \pi_4(G/\Cat) \longrightarrow  [M^\bullet, G/\Cat] \]
is onto because~$\pi_3(G/\Cat)=0$.  It follows that the induced homomorphism
\[ \phi^* \colon [M^\bullet, G/\Cat] \longrightarrow  \pi_6(G/\Cat) \]
vanishes since the composite~$C \circ \phi \colon S^6 \to \vee S^4$ must be a wedge of trivial maps or the essential map~$\eta^2 \colon S^6 \to S^4$
\footnote{In fact~$c \circ \phi$ is null-homotopic since~$M$ is a manifold and the homotopy class of~$c \circ \phi$ detects the exotic class of a 2-connected 7-dimensional Poincar\'{e} complex.} 
and the induced homomorphism~$(\eta^2)^* \colon \pi_4(G/\Cat) \to \pi_6(G/\Cat)$ factors through~$\pi_5(G/\Cat) = 0$.  From the short exact sequence
\[ 0=\pi_7(G/\Cat) \longrightarrow  [M, G/\Cat] \longrightarrow  [M^\bullet, G/\Cat] \stackrel{\phi^*}{\longrightarrow } \pi_6(G/\Cat)\]
we deduce that~$[M, G/\Cat] \to [M^\bullet, G/\Cat]$ is an isomorphism.
\end{proof}

Now note that~$H^3(M; \Z/2) = H^4(M; \Z/2) = 0$ if and only if~$H^4(M)$ contains no $2$-torsion.  In this case the forgetful map~$F \colon \mathcal{S}^{PL(M)} \to \mathcal{S}^{\Top}(M)$ is a bijection and~$M$ admits no exotic homeomorphisms.  

{\em Henceforth we assume that~$H^4(M)$ contains 2-torsion.}  We see that the~$PL$ normal invariants of exotic homeomorphisms lie in 
\[{\rm Ker}\bigl([M, G/PL] \stackrel F\longrightarrow [M, G/\Top]\bigr) \cong H^3(M; \Z/2) \; .\]
Using Lemma~\ref{C2.L2} we shall regard~$H^3(M; \Z/2)$ as a subset of~$\mathcal{S}^{PL}(M) \equiv [M, G/PL]$.

\begin{Lemma} \label{C2.L3}
Let~$x \in H^3(M; \Z/2)$.  Then there is a self-homotopy equivalence~$p(x) \colon M \simeq M$ such that:
\begin{enumerate}
\item\label{C2.L3.1} the PL normal invariant of~$p(x)$ is~$x \in [M, G/PL]$,
\item\label{C2.L3.2} $t_M(p(x)^*E) - t_M(E) = \bigl( x \smile c_2(E)  \bigr) [M]_2$ for all~$E \in \Bun(M)$.
\end{enumerate}
\end{Lemma}

Before proving Lemma~\ref{C2.L3} let us complete the proof of Theorem~\ref{C2.T1}.  Let~$h \colon N \simeq M$ be a homeomorphism.  If~$h$ is not exotic then~$t_M = t_N \circ h^*$ by Lemma~\ref{C2.L1}.  Conversely, suppose that~$h$ is exotic: the~$PL$ normal invariant of~$h$ is a non-zero element~$x \in H^3(M; \Z/2)$.  By Lemma~\ref{C2.L3}~\eqref{C2.L3.1}, we have~$[h] = [p(x)] \in \mathcal{S}^{PL}(M)$ and so by definition there is a~$PL$ homeomorphism~$d \colon N \to M$ such that~$h \simeq p(x) \circ d$.  It follows for all~$E \in \Bun(M)$ that
\[ t_N(h^*E) = t_N(d^*p(x)^*E) = t_M(p(x)^*E) = t_M(E) + \bigl( x \smile c_2(E) \bigr) [M]_2 \; .\]
Here we used Lemma~\ref{C2.L1} for the second equality and Lemma~\ref{C2.L3}~\eqref{C2.L3.2} for third equality.  Finally, note that~$p(x)$ has a trivial topological normal invariant and so is homotopic to an exotic self-homeomorphism~$h \colon M \cong M$.  This completes the proof of Theorem~\ref{C2.T1}.
\end{proof}

\begin{proof}[Proof of Lemma~\ref{C2.L3}]
We begin with the definition of~$p(x)$ and the proof of part~\eqref{C2.L3.1}.  Recall that~$M^\bullet := M - {\rm int}(D^7)$ and that~$M^\bullet$ is homotopy equivalent to the Moore space~$M(H, 3)$ where~$H = H_3(M) \cong \oplus_{i=1}^r C_{k(r)}$ is a sum of cyclic groups.  We assume that~$k(r)$ is even if and only if~$i \in \{ 1, \dots, s \}$.  

The self-equivalences~$p(x)\colon M \to M$ we construct will be {\em pinch maps}.  Pinch maps are described in~\cite[\S 4]{MTW} for manifolds with boundary but to begin we shall consider the closed case.  Take an element~$\psi \in \pi_7(M^\bullet)$ and define~$p(\psi)$ to be the composition
\[ p(\psi) := ({\rm Id} \vee \psi ) \circ p \colon M \longrightarrow  M \vee S^7 \longrightarrow  M \]
where~$p \colon M \to M \vee S^7$ collapses the boundary of a 7-disc~$D^7 \subset M$ to a point and we identify~$M = M/D^7$.

To build pinch maps~$p(\psi)$ we shall need some knowledge of~$\pi_7(M^\bullet)$: recall that~$C \colon M^\bullet \to \bigvee_{i=1}^rS^4_i$ collapses the $3$-skeleton and that~$\pi_7(S^4_i)  \cong \Z \oplus \Z/12$.  Observe also that~$M^\bullet = S ( M(H, 2) )$ is the suspension of the degree two Moore space.

\begin{Lemma} \label{C2.L4}
For~$i = 1, \dots, s$ there are homotopy classes~$\varphi_i \in \pi_6(M(H, 2))$ such that the suspensions~$\psi_i \colon = S(\varphi_i) \in \pi_7(M^\bullet)$ satisfy
\[ C_*(\psi_i) = (0, 6) \in \Z\times\Z/12\Z\cong\pi_7(S^4_i) \subset \pi_7\biggl(\bigvee_{i=1}^rS^4_i\biggr) \; . \]
\end{Lemma}

\begin{proof}
The stabilisation homomorphism~$S: \pi_6(M(H, 2) \to \pi_7(M^\bullet)$ fits into the following commutative diagram of exact sequences
\footnote{The lower exact sequence was comprehensively studied for~$r = 1$ in~\cite{S}. }.
%
\[
  \begin{CD}
      \pi_6\Bigl(\bigvee\limits_{i=1}^r S^2\Bigr)&\;\longrightarrow\;&\pi_6(M(H, 2))&\;\stackrel{j_*}\longrightarrow\;&\pi_6\Bigl(M(H,2), \bigvee\limits_{i=1}^rS^2\Bigr)&\;\stackrel{\del}\longrightarrow\;& \pi_5\Bigl(\bigvee\limits_{i=1}^r S^2_i\Bigr)\\
      @VVSV@VVSV@VVSV@VVSV\\
    \pi_7\Bigl(\bigvee\limits_{i=1}^r S^3\Bigr)&\;\longrightarrow\;&\pi_{7}(M^\bullet)&\;\stackrel{j_*}\longrightarrow\;&\pi_7\Bigl(M^\bullet, \bigvee\limits_{i=1}^rS^3\Bigr)&\;\stackrel{\del}\longrightarrow\;& \pi_6\Bigl(\bigvee\limits_{i=1}^r S^3_i\Bigr)\\
  \end{CD}
  \]
%
%
Let~$X_{4, i}$ generate the homotopy group~$\pi_4(S^3_i \cup_{k(i)} e^4_i, S^3) \cong \Z$, and let~$\Psi$ generate~$\pi_7(D^4, S^3) \cong \pi_6(S^3)\cong \Z/12$.
On classes of the form~$X_{4, i} \circ \Psi$, the boundary map~$\del$ is given by multiplication by~$k(i)$.  An analogous statement holds for~$X_{3, i}$, a generator of~$\pi_3(S^2_i \cup_{k(i)} e^3_i, S^2_i) \cong \Z$ and compositions~$X_{3, i} \circ \varPhi$ for~$\varPhi$ a generator of~$\pi_6(D^3, S^2) \cong \pi_5(S^2) \cong \Z/2$.  It is well known that the stabilisation map~$S \colon \pi_5(S^2) \to \pi_6(S^3)$ maps~$\varPhi$ to~$6 \Psi$.  It follows for~$i \leq s$, that the classes~$X_{3, i} \circ \varPhi$ and~$X_{4, i} \circ (6 \Psi)$ are in the image of the maps~$j_*$ and that~$S(X_{3, i} \circ \varPhi) = X_{4,i} \circ (6 \Psi)$.  We choose~$\varphi_i \in \pi_6(M(H, 2))$ such that
\[ j_*(\varphi_i) = X_{3, i} \circ \varPhi \]
and define
\begin{equation} \label{C2.E3a}
 \psi_i := S(\varphi_i)  \in \pi_7(M^\bullet)
 \end{equation}
so that in particular~$j_*(\psi_i) = X_{4, i} \circ (6 \Psi)$.  

To see that the homotopy classes~$\psi_i$ have the stated property, we pass to stable homotopy groups: denoted~$\pi_*^S$ with stabilisation~$S: \pi_* \to \pi_*^S$.  By excision we have an isomorphism
\[\pi_7^S\biggl(M^\bullet, \bigvee_{i=1}^rS^3_i\biggr) \cong \pi_7^S\biggl(\bigvee_{i=1}^rS^4_i\biggr) \]
and since~$X_{4, i}$ stabilises to generate~$\pi^S_4(S^3_i \cup_{k(i)} e^4_i, S^3_i) \cong \Z$ we quickly deduce that~$C_*S(\psi_i)$ is the element of order two in~$\pi_7^S(S^4_i)$.  The stabilisation map~$\pi_6(S^3) \to \pi_3^S$ is isomorphic to the inclusion~$\Z/12 \to \Z/24$ and this completes the proof.
\end{proof}

\begin{Definition} \label{C2.D1}
Using the homotopy equivalence $M^\bullet \simeq \bigvee_{i=1}^rM(\Z/k(i), 3)$ define generators $\{ x_1, \dots, x_s \}$ of $H^3(M; \Z/2)$ where each $x_i$ is the pullback of the generator of $H^3(M(\Z/k(i)); \Z/2) = \Z/2$.  Given~$x \in H^3(M; \Z/2)$ write $x = \Sigma_{i=1}^s \epsilon_i x_i$ where $\epsilon_i = 0$ or $1$ and define
\[ p(x) := p(\Sigma_{i=1}^s \epsilon_i \psi_i) \]
with~$\psi_i \in \pi_7(M^\bullet)$ defined as in \eqref{C2.E3a}.
\end{Definition}

To complete the proof of part \eqref{C2.L3.1} of Lemma~\ref{C2.L3} it remains to show that~$\eta^{PL}(p(x)) = x \in H^3(M; \Z/2) \subset [M, G/PL]$.  We shall apply the methods and results of~\cite{MTW} which require us to consider {\em tangential surgery theory}.  We shall not give the details here but rather refer to~\cite[\S 2]{MTW}.  Another feature of~\cite{MTW} is that they work with manifolds with boundary but this is fine in our setting: there are obvious maps~$\mathcal{S}^{\Cat}(M) \to \mathcal{S}^{\Cat}(M^\bullet)$ which in our case are bijections.  We shall also write~$p(\psi_i) \colon M^\bullet \to M^\bullet$ for the corresponding pinch map on~$M^\bullet$.

Choose a bundle isomorphism~$b(\psi_i) \colon \nu_{M^\bullet} \cong \nu_{M^\bullet}$ covering~$p(\psi_i)$ such that the pair~$(p(\psi_i), b(\psi_i))$ defines an element in the {\em tangential structure set} of~$M$, $\mathcal{S}^{\Cat, t}(M)$.  The tangential normal invariant set is in bijective correspondence with~$[M, G]$ and the obvious forgetful map fits into the following commutative square (see~\cite[(2.4)]{MTW}).
\begin{equation*}
  \begin{CD}
   \mathcal{S}^{\Cat, t}(M^\bullet) @>F>> \mathcal{S}^{\Cat}(M^\bullet)\\
    @VV\eta^{\Cat, t}V@VV\eta^{\Cat}V\\
    [M^\bullet, G]@>>>[M^\bullet, G/\Cat]
  \end{CD}
\end{equation*}
Moreover, for~$\Cat = PL$ we have that~$[M^\bullet, G] \to [M^\bullet, G/PL]$ is isomorphic to the boundary map~$\beta_{24} \colon H^3(M; \Z/24) \to H^4(M; \Z)$ in the cohomology Bockstein sequence for~$\Z \stackrel{\times 24}{\to} \Z \to \Z/24$.  It is a simple matter to check that these facts and the following lemma complete the proof of part~\eqref{C2.L3.1} of Lemma~\ref{C2.L3}.

\begin{Lemma} \label{C2.L5}
The tangential normal invariant of~$(p(\psi_i), b(\psi_i))$ is independent of~$b(\psi_i)$ and given by the equation
\[ \eta^{PL, t}(p(\psi_i), b(\psi_i)) = \iota(x_i) \in H^3(M; \Z/24) \cong [M^\bullet, G] \]
where  $x_i$ is as in Defintion \ref{C2.D1} and~$\iota \colon H^3(M; \Z/2) \to H^3(M; \Z/24)$ is induced by the inclusion of coefficients~$\Z/2 \hookrightarrow \Z/24$.
\end{Lemma}

\begin{proof}
The tangential normal invariant of~$(p, b) := (p(\psi_i), b(\psi_i))$ is computed as follows: first use the Pontrjagin-Thom isomorphism to obtain an element in the stable homotopy group of the pair~$(T(\nu_{M^\bullet}), T(\nu_{M^\bullet}|\ast))$ where~$T(\xi)$ denotes the Thom space of a vector bundle which we assume has rank~$k > 8$ and~$\ast \in \del M^\bullet$ is a base-point in the boundary of~$M^\bullet$.  We write:
\[ PT(p, b) \in \pi_{7+k}(T(\nu_{M^\bullet}), T(\nu_{M^\bullet}|\ast))\; . \]
Next one applies Spanier-Whitehead duality 
\[ D \colon \pi_{7+k}(T(\nu_{M^\bullet}), T(\nu_{M^\bullet}|\ast)) \longrightarrow  [M, G]\] 
and defines~$\eta^{Cat, t}(p, b) := D(PT(p, b))$.

Now~$M^\bullet = \Sigma M(H, 2)$ is a suspension and so~$T(\nu_{M^\bullet})/T(\nu_{M^\bullet}|\ast) \simeq \Sigma^k M^\bullet$ (see the remark at the bottom of~\cite[p.\,469]{MTW}).  Moreover, the pinch map~$p(\psi_i)$ is defined using a suspension~$\psi_i = S(\varphi_i)$, $\varphi_i \in \pi_6(M(H, 2))$.  It follows by~\cite[Theorem 4.7 and Lemma 4.8]{MTW} that 
\[ \eta^{PL, t}(p(\psi_i), b(\psi_i)) = D(S(\psi_i)) \]
where~$S(\psi_i) \in \pi_7^S(M^\bullet)$ is the stabilisation of~$\psi_i$.  But by the proof of Lemma~\ref{C2.L4} we know that~$C_*S(\psi_i) \in \pi_7^S(\bigvee_{i=1}^r(S^4_i))$ is precisely the element of order two in the summand~$\pi_7^S(S^4_i) \cong \Z/24$.   Applying Spanier-Whitehead duality to this statement completes the proof.
\end{proof}

We now prove part~$(2)$ of Lemma~\ref{C2.L5}.  If~$x = 0$ then~$p(x) \simeq {\rm Id}$ and the statement is obvious.  Assume then that~$x \neq 0 \in H^3(M; \Z/2)$.  By choosing an appropriate set of generators for~$H_3(M)$ we may assume that~$x = x_1$ in the notation of Definition \ref{C2.D1}.  Let~$f \colon M \to BS^3$ classify~$E \in \Bun(M)$.  As~$p(x_1)$ is the pinch map on~$\psi_1 \in \pi_7(M^\bullet)$ the induced map~$p(x_1)^*$ of~$\Bun(M)$ is given by 
\begin{equation} \label{C2.E2}
 p(x_1)^*E= E \conn F_E 
\end{equation}
where~$F_E \in \Bun(S^7)$ is classified by the composition
\[ f|_{M^\bullet} \circ \psi_1 \colon S^7 \longrightarrow  M^\bullet \longrightarrow  BS^3 \; . \] 
But~$BS^3$ is 3-connected and so~$f|_{M^\bullet}$ 
factors through the collapse map $C$ of \eqref{C2.E1}:
\[ f|_{M^\bullet} \; = \; \bar f \circ C \colon M^\bullet \longrightarrow \bigvee_{i=1}^rS^4_i \longrightarrow BS^3 \; .\]
By definition~$\psi_1 \colon S^7 \to M^\bullet$ is such that~$C_*(\psi_1) = (0, 6) \in \pi_7(S^4_1) \cong \Z \oplus \Z_{12}$.  Recall that~$H \to \H P^\infty$ is the tautological bundle and so~$F_E = (f|_{M^\bullet} \circ \psi_1)^*H$ is determined by~$c_2(E)$ as follows:   
\begin{equation} \label{C2.E3}
 F_E = \left\{ \begin{array}{ll} 6 \in \pi_7(BS^3) \cong \Z/12\Z \, , & \text{ if~$c_2(E)[S^4_1]$ is odd\;,}\\ \\
0 \in \pi_7(BS^3) \cong \Z/12\Z \, , & \text{if~$c_2(E)[S^4_1]$ is even\;.} \end{array} \right. 
\end{equation}
%
Applying Theorem \ref{C1.T2}~\eqref{C1.T2.1} and Proposition~\ref{B3.P1}~\eqref{B3.P1.2} to \eqref{C2.E2} and \eqref{C2.E3} above we see that
\[ t_M(p(x_1)^* E) - t_M(E) = t_{S^7}(F_E) = \bigl( x_1 \smile c_2(E) \bigr)[M]_2 \in \Z/2 \subset \Q/\Z \;. \quad\qedhere\] 
\end{proof}

\subsection{Remarks on simply-connected 7-manifolds} \label{C3}
In this subsection we make some remarks about the role of the $t$-invariant in the classification of simply connected spin $7$-manifolds.  

\begin{Remark}\label{B1.R2}
  The invariant~$t_M$ is related to the Kreck-Stolz invariants~$s_2$ and~$s_3$ of~\cite{KS}.
  To begin let~$M$ be a closed spin $(4k-1)$-manifold as in~\ref{B2.D1} and
  let~$L\to M$ be a Hermitian line bundle with~$c_1(L)=a\in H^2(M)$.
  In analogy with~\eqref{B1.2},
  we define a characteristic class~$\ch'$ of complex line bundles such that
	$$\ch(L)=1+c_1(L)+c_1(L)^2\,\ch'(L)\;.$$
  Assume that~$L$ extends to~$\overline L \to W$,
  where~$W$ is a compact spin manifold with~$M=\partial W$.
  The class~$c_1(\overline L)^2\in H^4(W)$ lifts uniquely to~$\bar v\in H^4(W,M;\Q)$,
  so we can define
	$$s_M(a) : = \bigl(\widehat A(TW)\,\ch'(\overline L)\,\bar v\bigr)[W,M]\in\Q/\Z$$
  independent of the choices of~$W$ and~$\overline L$.
  If~$k = 2$ so that $M$ is $7$-dimensional and $H^2(M)\cong \Z$ with generator~$z$,
  then Kreck and Stolz considered~$s_2(M) :=s_M(z)$
  and~$s_3(M) := s_M(2z)$.

A Hermitian metric on~$L$ identifies the dual bundle~$L^*$ with the complex
conjugate of~$L$, and this identification is parallel with respect to
compatible Hermitian connections.
Thus the bundle 
\[ E := L \oplus L^* \]
carries a natural quaternionic structure
  with~$c_2(E)=-c_1(L)^2$. 
  In particular, we have
	$$\ch(E)=\ch(L)+\ch(L^*)\qquad\text{and}\qquad
	\ch'(E)=\ch'(L)+\ch'(L^*)$$
  and~$\bar v=-\bar c_2(E)$.
  Thus~$a_{k+1}t_M(E)=2s_M(a)$, or equivalently
	$$t_M(E)=a_k\,s_M(a) \; .$$
  Hence the~$t$-invariant generalises~the $s$-invariant if~$k$ is even.
  In particular for~$k=2$, $t_M$ generalises~$s_2$ and~$s_3$.
  There is an analogous argument comparing the intrinsic definitions
  of~$s_M$ and~$t_M$.
\end{Remark}

We may now unify the classification results of ~\cite{KS} and~\cite{Crow}.  Let~$N$ and~$M$ be closed smooth simply-connected spin $7$-manifolds and assume that~$\pi_2(N) \cong \pi_2(M)$ is either trivial or infinite cyclic.  In the latter case assume that~$H^4(N) \cong H^4(M)$ are finite cyclic groups generated by the square of a generator of~$H^2$ and extend the definition of the $t$-invariant to the~$\hat t$-invariant which is the pair
\[ \hat t_M = (t_M, A_M) \]
where~$A_M \colon H^2(M)\cong\Bun_{S^1}(M) \to \Bun(M)$ maps~$L$ to~$L \oplus L^*$ and~$t_M$ is the $t$-invariant.

\begin{Theorem} \label{C3.T1}
Let~$N$ and~$M$ be simply connected spin $7$-manifolds as above.  Then~$N$ is diffeomorphic to~$M$ if and only if~$\hat t_N$ is isomorphic to~$\hat t_M$ and~$\mu(N) = \mu(M)$.
\end{Theorem}


\begin{Remark}\label{B1.R3}
Expanding on~\cite[Theorem 6]{Kreck}, Hepworth~\cite[Thereom 2.2.9]{Hep} gave a classification theorem for simply connected compact spin~$7$-manifolds
  with~$H^3(M)=0$,
  such that~$H^4(M)$ is torsion and generated by~$p_M = \frac{p_1}2(TM)$
  and cup products of elements of~$H^2(M)$.
  
  The group~$H^2(M)$ is necessarily free, so there exists a $\Z$-basis~$(x_1,\dots,x_r)$.
  Apart from the Eells-Kuiper invariant,
  and the triple Massey products,
  Hepworth defines five more families of invariants,
  $\sigma_i$, $\sigma_{ij}$, $\tau_i$, $\tau_{ij}$, and~$\tau_{ij,k}$
  for~$i$, $j$, $k\in\{1,\dots,r\}$,
  based on the existence of a closed spin manifold~$W$ with~$\partial W=M$ such that the inclusion~$M \to W$ induces an isomorphism~$H^2(W) \cong H^2(M)$.
  Together,
  these invariants give a complete set of diffeomorphism invariants.
  The $\sigma$- and $\tau$-invariants can be expressed in terms
  of the function~$s_M$ of Remark~\ref{B1.R2} and some Massey products.
  As above, we may use~$t_M(L\oplus L^*)$ instead of~$s_M(L)$.
  Note that Hepworth also uses the linking form,
  but the linking form can be recovered from the $t$-invariant just as in Theorem~\ref{C1.T2}.

If we now drop the assumption that ~$H^4(M)$ is generated by~$p_M$
  and products of elements of~$H^2(M)$, then 
  as in the highly connected case,
  the invariant~$t_M$ contains potentially more information than~$s_M$
  because one can define~$t_M(E)$ even if~$c_2(E)$ is not
  in the span of~$H^2(M) \smile H^2(M)$.
  It therefore seems reasonable to hope that simply connected spin $7$-manifolds~$M$ with~$H^3(M) = 0$ and~$H^4(M)$ finite can be classified via their Eells-Kuiper invariants, their Massey product structure, and their~$\hat t$-invariants.
\end{Remark}

\section{Bundles over \texorpdfstring{$\H P^k$}{HPk}} \label{D}

In this section we investigate the map~$c_2 \colon \Bun(\H P^k) \to H^4(\H P^k)$ using the $t$-invariant.  Determining the image of~$c_2$ is a difficult problem already investigated in~\cite{FG, GS} and elsewhere.  One calls an integer~$c$ {\em $k$-realisable} if there is a bundle~$E \in \Bun(\H P^k)$ such that
\[ c_2(E) = c \cdot c_2(H) \]
where we recall that~$H$ is the tautological bundle over~$\H P^k$.  Below we use the $t$-invariant to give a new proof of the necessary criteria of~\cite{FG} for $k$-realisability.

Consider the Hopf fibration~$p_k\colon S^{4k-1}\to\H P^{k-1}$
and let~$\bar p_k\colon W_k\to\H P^{k-1}$ denote the corresponding $D^4$-bundle
with boundary~$S^{4k-1}$.
Then
	$$\H P^k=W_k\cup_{S^{4k-1}}D^{4k}\;.$$
One may therefore consider the problem of building bundles over~$\H P^k$ inductively: assume that
$E_{k-1}$ is a quaternionic line bundle on~$\H P^{k-1}$,
then the pullback~$\bar p_k^*E_{k-1}$ on~$W$ extends to~$E_k\to\H P^k$
iff its restriction~$\bar p_k^*E_{k-1}|_{S^{4k-1}}=p_k^*E_{k-1}$ to the boundary
is trivial.  
If~$p_k^*E$ is trivial then the group~$\pi_{4k-1}(S^3)$ acts transitively on the set of possible extensions of~$E$ to~$\H P^k$.  The situation is summarised in the following commutative diagram.


$$
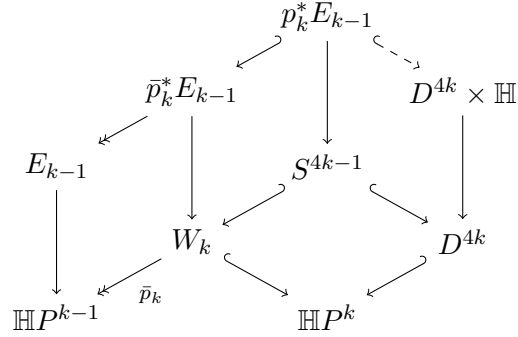
\begin{tikzpicture}
    \node (E) at (0,2) {$E_{k-1}$} ;
    \node (Ebp) at (1.8, 3) {$\bar p_k^*E_{k-1}$} ;
    \node (Ep) at (3.6,4) {$p_k^*E_{k-1}$} ;
    \node (Ek) at (5.4,3) {$D^{4k}\times\H$} ;
    \node (W) at (1.8, 1) {$W_k$};
    \node (Hx) at (0,0) {$\H P^{k-1}$} ;
    \node (Sx) at (3.6,2){$S^{4k-1}$} ;
     \node (D) at (5.4, 1) {$D^{4k}$} ;
    \node (Hy) at (3.6, 0) {$\H P^k$} ;
\draw[->] (E) -- (Hx) ;
\draw[->] (Ebp) -- (W) ;
\draw[->>] (Ebp) -- (E) ;
\draw[->] (Ep) -- (Sx) ;
\draw[left hook->] (Ep) -- (Ebp) ;
\draw[right hook->] (Sx) -- (D) ;
\draw[left hook->] (Sx) -- (W) ;
\draw[left hook->] (D) -- (Hy) ;
\draw[->>] (W) -- node[below right] {$\scriptstyle \bar p_k$} (Hx) ;
  \draw[right hook->] (W) -- (Hy) ;
\draw[->] (Ek) -- (D) ;
\draw[dashed,right hook->] (Ep) -- (Ek) ;
\end{tikzpicture}$$

Determining whether $p_k^*E_{k-1}$ is trivial is a difficult problem in general.  However Theorem \ref{D1.T1} below determines the $t$-invariant of such pull-backs in terms of $c_2(E_{k-1})$.  For ease of computations with signs we first define~$a : = -c_2(H)$ to be the other generator of~$H^4(\H P^\infty) \cong \Z$ and for all $j$ we identify $H^4(\H P^j) = H^4(\H P^\infty)$ in the obvious way.

\begin{Theorem}\label{D1.T1}
  If~$k\ge 2$ and there exists a quaternionic line bundle~$E_c$ on~$\H P^{k-1}$
  with~$c_2(E_c)= -c \cdot a$ in~$H^4(\H P^{k-1})\cong\Z a$,
  then
  	$$t_{S^{4k-1}}\bigl(p_k^*E_c\bigr)
	=\frac{a_k}{(2k)!}\prod_{j=0}^{k-1}(c-j^2)
	\quad\in\quad\Q/\Z\;.$$
\end{Theorem}

We defer the proof Theorem~\ref{D1.T1} to the end of the section. 

\begin{Corollary}[\cite{FG}{Theorem 1.1}] \label{D1.C1}
If an integer~$c$ is $k$-realisable then
  \begin{equation*}
    \frac{a_j}{(2j)!}\prod_{i=0}^{j-1}(c-i^2) = 0 \in\Q/\Z
	\qquad\text{for all }2\le j\le k\;.
  \end{equation*}
\end{Corollary}

\begin{proof}
Let~$E_c \in \Bun(\H P^k)$ be a bundle with~$c_2(E_c) =  c \cdot c_2(\tau) = -c \cdot a$.  Then for all~$2 \leq j \leq k$, the bundle~$E_c^j : = p_j^*(E_c|_{\H P^{j-1}})= E_c|_{S^{4j-1}}$ is trivial and in particular~$t_{S^{4j-1}}(E_c^j) = 0$.  Now apply Theorem~\ref{D1.T1}.
\end{proof}

\begin{Remark}\label{D1.R1}

 Feder and Gitler proved the that the conditions of Corollary~\ref{D1.C1} are necessary using different methods in~\cite{FG}.  They also gave a proof that the conditions are sufficient if~$k=\infty$ 
 \footnote{In~\cite{FG} Feder and Gitler quote Sullivan who credits this statement to unpublished work of I.~Bernstein, R.~Stong, L.~Smith and G.~Cooke.},
  where they are satisfied iff~$c$ is an odd square or zero.  Feder and Gitler strongly suggested that their conditions are also sufficient for finite~$k$, but so far, this has been proved only for~$k=2$, $3$, $4$, $5$ (see~\cite[\S 2]{GS} for more details on this issue).
  Moreover, the bundle~$E_c$ is not unique in general, however Gon\c calves and Spreafico~\cite[Theorem 3]{GS} proved that  the number of non-isomorphic
  bundles with the same second Chern class is finite and depends
  only on~$k$ and the parity of~$c$.  
\end{Remark}

\begin{Remark} \label{D1.R2}
Theorem~\ref{D1.T1} and Theorem~\ref{B3.T1} shed some light on the Feder-Gitler conditions: as we saw above, the {\em a priori\/} obstruction to extending a bundle~$E_c \to \H P^{j-1}$ to~$\H P^j$ is the bundle~$p_j^*E_c \in \Bun(S^{4j-1}) \cong \pi_{4j-2}(S^3)$.  We see that the claim that the Feder-Gitler conditions are sufficient for~$c$ to be $k$-realisable is equivalent to the following claim:  For each~$2 \leq j \leq k$ such that~$c$ is $(j-1)$-realisable, the set 
\[ t_{S^{4j-1}}^{-1}(0) \; \cap \; p_j^* \bigl( c_2^{-1} ( c \cdot c_2(H)) \bigr) \subset \Bun(S^{4j-1}) \]
contains the trivial bundle $S^{4j-1} \times \H$.  By Theorem \ref{D1.T1}, Remark \ref{B3.R3} and Corollary~\ref{B3.C1} this is true if~$k = 2$ or~$3$ but becomes surprising if $k > 3$.  By Proposition~\ref{B3.P1} the $t$-invariant measures only minus the $e$-invariant of the stabilisation~$S(p_j^*E_c) \in \pi_{4k-5}^S$ and for $k > 3$ the stabilisation homomorphism~$S \colon \pi_{4k-2}(S^3) \to \pi_{4k-5}^S$ and the $e$-invariant~$e \colon \pi_{4k-5}^S \to \Q/\Z$ both typically have non-trivial kernels.
\end{Remark} 

We next discuss Theorem~\ref{D1.T1} for~$k = 2, 3$ and~$4$.

\begin{Example}\label{D1.K2}
  For~$k=2$,
  we recover a special case of Proposition~\ref{S4.T1}.
  In particular,
  there exist quaternionic line bundles~$E_c$ on~$\H P^1=S^4$
  with~$c_2(E_c)= - ~c\cdot~a$
  for all~$c\in\Z$,
  and we have
	$$t_{S^7}\bigl(p_2^*E_c\bigr)=\frac{c(c-1)}{24}
	\quad\in\quad\frac1{12}\Z\bigm/\Z\;.$$
	This confirms Theorem~\ref{C1.T2} (1) and Remark~\ref{B3.R3} stating that~$t_{S^7}$ is injective.
\end{Example}

\begin{Example}\label{D1.K3}
  A bundle~$p_2^*E_c$ on~$S^7$ is trivial iff~$c(c-1)\equiv 0$ mod~$24$.
  This is the case iff
  \begin{equation*}
    c\equiv 0,1\mod 8\qquad\text{and}\qquad c\equiv 0,1\mod 3\;.
  \end{equation*}
  Because~$\pi_7(S^3)\cong\Z_2$,
  there exist at most two extensions of~$E_c$ to~$\H P^2$
  up to isomorphism.
  We pick one for each~$c$ and still denote it by~$E_c$.

  By Theorem~\ref{D1.T1},
	$$t_{S^{11}}\bigl(p_3^*E_c\bigr)
	=\frac{2c(c-1)(c-4)}{720}
	=\frac{c(c-1)}{24}\cdot\frac{c-4}{15}
	\quad\in\quad\frac1{15}\Z\bigm/\Z\;,$$
  because the first factor is always an integer.
  Because~$E_c$ exists on~$\H P^2$ for~$c=33$ and
	$$t_{S^{11}}\bigl(p_3^*E_{3}\bigr)=\frac1{15}\;,$$
  we conclude from Proposition~\ref{B3.P1}~\eqref{B3.P1.2}
  that the $t$-invariant again gives an isomorphism
	$$ t_{S^{11}}\colon\Bun(S^{11})\cong\pi_{10}(S^3)
	\longrightarrow\Z/15\Z\;.$$
	This confirms Corollary~\ref{B3.C1} stating that~$t_{S^{11}}$ is injective.

\end{Example}

\begin{Example}\label{D1.K4}
  A bundle~$p_3^*E_c$ exists and is trivial on~$S^{11}$ iff
  \begin{equation*}
	c\equiv 0,1\mod 8\;,\qquad
	c\equiv 0,1,4,7\mod 9\;,
	\qquad\text{and}\qquad c\equiv 0,1,4\mod 5\;.
  \end{equation*}
  In this case,
  it extends to~$\H P^3$.
  Note that because~$\pi_{11}(S^3)\cong\Z/2\Z$
  and because there were possibly two non-isomorphic bundles~$E_c$ on~$\H P^2$
  to start with,
  there can be up to four quaternionic line bundles on~$\H P^3$
  with second Chern class~$-c\cdot~a\in H^4(\H P^3)$.
  We pick one and again denote it by~$E_c$.

  By Theorem~\ref{D1.T1},
	$$t_{S^{15}}\bigl(p_4^*E_c\bigr)
	=\frac{c(c-1)(c-4)(c-9)}{2^7\cdot 3^2\cdot 5\cdot 7}
	\quad\in\quad\frac1{28}\Z\bigm/\Z\;,$$
  because the numerator is always divisible by~$2^5\cdot 3^2\cdot 5$
  for all~$c\in\Z$ that can occur.
  For~$c=40$, we have
	$$t_{S^{15}}\bigl(p_4^*E_{40}\bigr)=\frac5{28}\;,$$
  which generates~$\frac1{28}\Z\bigm/\Z$.
  
  Now by~\cite[(7.14), Theorem 13.9, (13.6)']{T} the stabilisation homomorphism~$S \colon \pi_{14}(S^3) \to \pi_{11}^S$ is isomorphic to the homomorphism
\[ \Z/84 \oplus \Z/2 \oplus \Z/2 \stackrel{(\times 36, 0, 0)}{\xrightarrow{\hspace*{1.25cm}}} \Z/504\;,  \]
and by~\cite{A} the $e$-invariant~$e \colon \pi_{11}^S \to \Q/\Z$ is injective.  Moreover, by Remark~\ref{D1.R1} the condition~$c(c-1)(c-4)(c-9)\equiv 0$ mod~$\frac{8!}2$ is
  sufficient for the existence of a bundle~$E_c$ on~$\H P^4$,
  hence for the existence of a bundle~$p_4^*E_c$ on~$S^{15}$
  that is trivial.
  Thus the $3$-torsion component in~$\Bun(S^{15})$
  cannot be generated by bundles of the type~$p_4^*E_c$.
    \end{Example}


\begin{proof}[Proof of Theorem~\ref{D1.T1}]
  We compute~$t_{S^{4k-1}}$ using the zero bordism~$W_k$
  as in Definition~\eqref{B1.D1}.
  Using the homotopy equivalence~$\bar p_k \colon W_k \to \H P^k$
  to identify~$H^4(W_k) = H^4(\H P^{k-1})$ and the Thom isomorphism,
  we know that the generator~$a \in H^4(W_k)$
  lifts to a generator of~$H^4(W_k,S^{4k-1})\cong a\Z[a]/a^{k+1}\Z[a]$,
  In particular
	$$\bar c_2( E_c )\,\ch'(E_c)=2-2\,\cosh\bigl(\sqrt{ca}\bigr)
	\quad\in\quad H^\bullet(W_k,S^{4k-1};\Q)\cong a\Q[a]/a^{k+1}\Q[a]\;.$$

  We write~$H_{\R}$ for~$H$ regarded as a real vector bundle.  For the total Pontrjagin class of~$H_{\R}$
we obtain
	$$p(H_{\R})=p(H)^2=(1+a)^2\;.$$
Now let~$\End_\H(H)$ denote the real $4$-plane bundle of quaternionic bundles maps from~$H$ to itself.  The Pontrjagin classes and Euler class of~$\End_\H(H)$ are given by
	$$p_1(\End_\H(H))=4a
		\qquad\text{and}\qquad
	e(\End_\H(H))=p_2(\End_\H(H))=0\;.$$

The tangent bundle of~$\H P^k$ satisfies the relation
	$$T\H P^k\oplus \End_\H(H) \cong H^{\oplus(k+1)}\;.$$
\noindent
In particular, we have
\begin{align*}
	\widehat A(T\H P^k)
	&=\widehat A(H)^{k+1}\smile\widehat A(\End_\H(H))^{-1}\\
	&=\biggl(\frac{\sqrt a/2}{\sinh(\sqrt a/2)}\biggr)^{2k+2}
	\,\biggl(\frac{\sqrt a}{\sinh\sqrt a}\biggr)^{-1}\\
	&=\biggl(\frac{\sqrt a/2}{\sinh(\sqrt a/2)}\biggr)^{2k+1}
	\,\cosh\bigl(\sqrt a/2\bigr) \\
	&\in\quad H^\bullet(\H P^k;\Q)\cong\Q[a]/a^{k+1}\Q[a]\;,
\end{align*}
and we know that~$\widehat A(W_k)=\widehat A(T\H P^k)|_{W_k}$.

  Let~$\gamma$ be a sufficiently small contour around~$0$ in~$\C$,
  then~$\gamma^2$ goes around~$0$ twice.
  Note that~$a_k\,a_{k+1}=2$ for all~$k$.
  From the formula for~$\widehat A(T\H P^k)$ above,
  we compute
  \begin{align*}
    t_{S^{4k-1}}\bigl(p_k^*E_c\bigr)
    &=-\frac1{a_{k+1}}\,\bigl(\widehat A(TW_k)\,\bar c_2(\bar p_k^*E_c) \ch'(\bar p_k^*E_c) \bigr)
	[W_k,S^{4k-1}]\\
    &=\frac{a_k}{2\pi i}\int_\gamma
	\frac{\cosh(\sqrt a/2)\,\bigl(\cosh(\sqrt{ac})-1\bigr)}
		{2^{2k+1}\,\sqrt a\,\sinh(\sqrt a/2)^{2k+1}}\,da\\
    &=\frac{a_k}{2\pi i}\int_{\gamma^2}
	\frac{\cosh(\sqrt a/2)\,\bigl(\cosh(\sqrt{ac})-1\bigr)}
		{2^{2k+2}\,\sqrt a\,\sinh(\sqrt a/2)^{2k+1}}\,da\\
    &=\frac{a_k}{2\pi i}\int_\gamma
	\frac{\cosh\bigl(2\sqrt c\cdot\arsinh\frac z2\bigr)-1}{z^{2k+1}}\,dz
	\quad\in\quad\Q/\Z\;.
  \end{align*}
  In the last step,
  we have substituted~$z$ for~$2\sinh(\sqrt a/2)$.
  Note that the contour for~$z$ closes only after~$a$ has completed
  two cycles around the origin.

  It is well known that there exist even polynomials~$P_n$
  of degree at most~$2n$ such that
	$$\cosh(2nx)=P_n(\sinh x)=\sum_{k=0}^{n}p_{n,k}\,\sinh^{2k}x\;,$$
  and~$p_{n,n}=2^{2n-1}$.
  On the other hand,
	$$\bar t_k(y)
	:=\frac1{2\pi i}\int_\gamma
		\frac{\cosh\bigl(2y\cdot\arsinh\frac z2\bigr)-1}
			{z^{2k+1}}\,dz$$
  is an even polynomial of degree at most~$2k$ in~$y$
  such that~$\bar t_k(n)=2^{-2k}\,p_{n,k}$.
  Because~$p_{n,k}=0$ if~$n<k$,
  we conclude that~$\bar t_k$ vanishes at~$n=1-k$, \dots, $k-1$,
  and that~$\bar t_k(k)=2^{-2k}p_{k,k}=\frac12$.
  Since~$\bar t_k$ is even,
  this leads to the equations
  \begin{equation*}
    \bar t_k(y)
    =\frac12\prod_{j=0}^{k-1}\frac{(y^2-j^2)}{(k^2-j^2)}
    =\frac1{(2k)!}\,\prod_{j=0}^{k-1}(y^2-j^2)\;.
  \end{equation*}
  The proposition follows because~$t_{S^{4k-1}}\bigl(p_k^*E_c\bigr)
  =a_k\,\bar t_k\bigl(\sqrt c\bigr)$ mod~$\Z$.
\end{proof}

\bibliographystyle{alpha}

\end{document}